%% file: tk68-no_lines.tex
\documentclass[12pt]{amsart}
\usepackage{amsmath}
\usepackage{amssymb, latexsym, amsthm,bm}
\usepackage{cite}
\usepackage{mathscinet}
\usepackage{cancel}
\usepackage{xspace}
\usepackage{xcolor}
\usepackage{hyperref}
\usepackage{mathrsfs}  
\definecolor{modra3}{rgb}{.1,.0,.4}
\hypersetup{colorlinks=true, linkcolor=modra3, urlcolor=modra3, citecolor=modra3, pdfpagemode=UseNone, pdfstartview=}

\usepackage{url}
\usepackage{longtable}
\usepackage[obeyFinal,colorinlistoftodos,textsize=tiny]{todonotes}
\usepackage{tikz}
\usepackage[font=footnotesize]{caption}
\usepackage{mathtools}
\usepackage{array}
\usepackage{tikz}
\usepackage{stmaryrd}

\makeatletter
\@namedef{subjclassname@2020}{%
  \textup{2020} Mathematics Subject Classification}
\makeatother

\def\indep{{\mathcal{I}}}
\def\aa{{\mathscr{A}}}
\def\ceil#1{{\lceil{#1}\rceil}}

\def\Tg{{{\text{\rm tg}}}}

\def\ceil#1{\lceil{#1}\rceil}

\newcommand*{\Scale}[2][4]{\scalebox{#1}{\ensuremath{#2}}}%

\def\OL#1{{\overline{#1}}}

\def\HH{{\mathcal H}}
\def\XX{{\mathcal X}}
\def\ZZ{{\mathbb Z}}

\newtheorem{theorem}{Theorem} %Defines \begin{theorem} to write "Theorem"
\newtheorem{theorem*}{Theorem} %Defines \begin{theorem} to write "Theorem"
\newtheorem{proposition}[theorem]{Proposition} %square bracket tells it to count props with theorems
\newtheorem{corollary}[theorem]{Corollary}
\newtheorem{lemma}[theorem]{Lemma}
\newtheorem{conjecture}[theorem]{Conjecture}

\newtheorem{question}[theorem]{Question}

\definecolor{modra}{rgb}{0,0,.8}
\definecolor{seda}{rgb}{.7,.7,.7}
\definecolor{piros}{rgb}{.8,0,0}
\definecolor{zelen}{rgb}{0,.5,0}
\if 11    %11=true  10=false
\def\jk#1{{\color{modra} {\sc JK: }{\sf #1} }}
\def\gs#1{{\color{zelen} {\sc JS: }{\sf #1} }}
\else
\def\jk#1{}
\def\gs#1{}
\fi

\begin{document}

% *****************************************************************************

\title[Thrackles on nonplanar surfaces]{Thrackles on nonplanar surfaces}

% *****************************************************************************

\author[Hern\'andez-V\'elez]{C\'esar Hern\'andez-V\'elez}
\address{Facultad de Ciencias, Universidad Aut\'onoma de San Luis Potos\'{\i}, SLP 78000, Mexico}
\email{\tt cesar.velez@uaslp.mx}

\author[Kyn\v{c}l]{Jan Kyn\v{c}l}
\address{Department of Applied Mathematics, Faculty of Mathematics and Physics, Charles University, Prague, Czech Republic}
\email{\tt kyncl@kam.mff.cuni.cz}

\author[Salazar]{Gelasio Salazar}
\address{Instituto de F\'\i sica, Universidad Aut\'onoma de San Luis Potos\'{\i}, SLP 78000, Mexico}
\email{\tt gelasio.salazar@uaslp.mx}


% *****************************************************************************

%\subjclass[2020]{Primary 57K10; Secondary 57M15, 05C10}

\date{\today}
\def\W#1{{\widehat{#1}}}
\def\solv#1#2{{{#1}\rightarrowtail{#2}}}
\def\Th#1{{S_{g_{#1}}^{\bullet\bullet}}}

\begin{abstract}
A {\em thrackle} is a drawing of a graph on a surface such that (i) adjacent edges only intersect at their common vertex; and (ii) nonadjacent edges intersect at exactly one point, at which they cross. Conway conjectured that if a graph with $n$ vertices and $m$ edges can be thrackled on the plane, then $m\le n$. Conway's conjecture remains open; the best bound known is that $m\le 1.393n$. {Cairns and Nikolayevsky extended} this conjecture to the orientable surface $S_g$ of genus $g > 0$, claiming that if a graph with $n$ vertices and $m$ edges has a thrackle on $S_g$, then $m \le n + 2g$. We disprove this conjecture. In stark contrast with the planar case, we show that for each $g>0$ there is a connected graph with $n$ vertices and {$2n + 2g -8$} edges that can be thrackled on $S_g$. This leaves relatively little room for further progress involving thrackles on orientable surfaces, as every connected graph with $n$ vertices and $m$ edges that can be thrackled on $S_g$ satisfies that $m \le 2n + 4g - 2$. We prove a similar result for nonorientable surfaces. {We also derive nontrivial upper and lower bounds on the minimum $g$ such that $K_{m,n}$ and $K_n$ can be thrackled on $S_g$.}
\end{abstract}

\maketitle

\section{Introduction}\label{sec:intro}

All graphs under consideration are simple, that is, they have neither loops nor parallel edges. We use $S_g$ (respectively, $N_g$) to denote a compact orientable (respectively, nonorientable) surface of genus $g > 0$. Thus, for instance, $S_2$ is a double torus and $N_2$ is a Klein bottle. We recall that if $M$ and $M'$ are surfaces, then $M \# M'$ denotes the connected sum of $M$ and $M'$.

A {\em thrackle} is a drawing of a graph on some surface $S$ such that (i) adjacent edges only intersect at their common vertex; and (ii) nonadjacent edges intersect at exactly one point, at which they cross. If such a drawing exists, then the graph can be {\em thrackled} on $S$.

Let $G$ be a graph with $n$ vertices and $m$ edges. Conway conjectured that if $G$ can be thrackled on the plane, then $m \le n$.  Conway's conjecture has been proved only for some particular cases~\cite{miserehthrackles0,janosthrackles,cairnsouterplanar}, and several aspects on the structure of thrackles have been investigated~\cite{miserehthrackles,woodallthrackles,oswinthrackles}. Lov\'asz, Pach, and Szegedy proved that $m\le 2n-3$~\cite{lovaszthrackles}. This bound has been refined several times~\cite{goddynxu,fulekthrackles,cairnsbounds,fulekthrackles2}. The best bound currently known is $m \le 1.393n$, recently established by Xu~\cite{xuthrackles}. 

\subsection{Thrackles on higher genus surfaces}
Much less is known for thrackles on nonplanar surfaces. Croft, Falconer, and Guy~\cite{croftthrackles} noted that the maximum possible value of $m - n$ should depend on the genus of the surface, but did not offer any explicit conjecture. 

{The first exhaustive investigation of thrackles on surfaces of higher genus was carried out by Cairns and Nikolayevsky~\cite{cairnsbounds}. In that paper they described a construction that shows the following.}

\begin{lemma}[\hglue -0.00001 cm{\cite[Section 4]{cairnsbounds}}]\label{lem:cn}
{If a graph with $n$ vertices and $m$ edges can be thrackled on a surface $M$, then there is a graph with $n+2$ vertices and $m+4$ edges that can be thrackled on $M\# S_1$.}
\end{lemma}

{Their construction is illustrated in Figure~\ref{fig:930}. We start by taking an arbitrary edge $st$ of $G$. By performing a self-homeomorphism on $M$, if necessary, without loss of generality we may assume that there is a rectangle $Q$ that contains $st$ and such that (i) $Q$ contains a connected part of each edge of $G$; (ii) the only vertices contained in $Q$ are $s$ and $t$; and (iii) the only crossings contained in $Q$ are the crossings of $st$. We may assume that the edges incident with $s$ cross the top side of $Q$, the edges incident with $t$ cross the bottom side of $Q$, and each edge that crosses $st$ intersects $Q$ on its left side and on its right side.}

% *********************************************************
\def\inca{{\Scale[3]{\text{\rm (a)}}}}
\def\incb{{\Scale[3]{\text{\rm (b)}}}}
\def\te#1{{\Scale[4.2]{#1}}}
\def\tf#1{{\Scale[4]{#1}}}
\def\tg#1{{\Scale[5.5]{#1}}}
\def\otg#1{{\Scale[6.5]{#1}}}
\def\somea{{\Scale[4]{\text{\rm (a)}}}}
\def\someb{{\Scale[4]{\text{\rm (b)}}}}
\def\Za{{\Scale[4.0]{E_{u\OL{v}}}}}
\def\Zb{{\Scale[4.0]{E_{\OL{u}\OL{v}}}}}
\def\Zc{{\Scale[4.0]{E_{\OL{u}{v}}}}}
\def\Rec{{\Scale[6.0]{\rho}}}
\begin{figure}[ht!]
\centering
\scalebox{0.13}{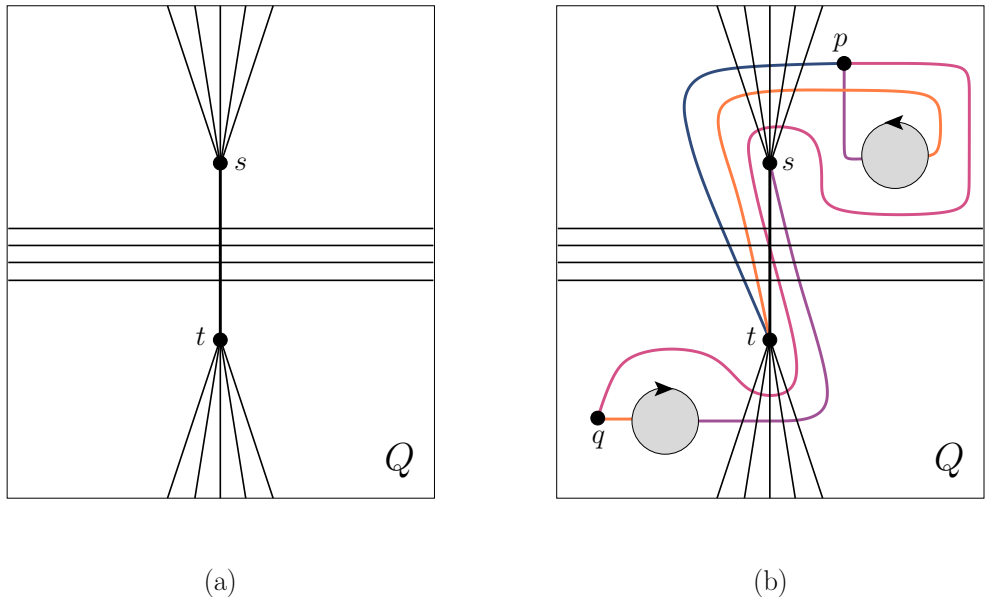}
\caption{Illustration of the Cairns--Nikolayevsky construction. In (b) the handle is obtained by removing the gray open disks and identifying their boundary circles using the orientations shown.}
\label{fig:930}
\end{figure}
% *********************************************************

{Thus, without loss of generality, we may assume that the layout is as illustrated in Figure~\ref{fig:930}(a). As we illustrate in Figure~\ref{fig:930}(b), with the inclusion of a new handle it is possible to add two vertices $p$ and $q$ and four edges $pq, pt, ps$, and $qt$, while keeping the thrackle property. Thus we obtain a thrackle on $M\# S_1$ of a graph with two more vertices and four more edges than $G$, as claimed in Lemma~\ref{lem:cn}.}

{It is well-known that every odd cycle can be thrackled on the plane (equivalently, the sphere); see for instance the thick $7$-cycle in Figure~\ref{fig:700}. A recursive use of Lemma~\ref{lem:cn}, taking as starting point a planar thrackle of any odd cycle, implies that for any odd integer $n\ge 3$ and every integer $g \ge 0$ there exist graphs with $n$ vertices and $n + 2g$ edges that can be thrackled on $S_g$. Cairns and Nikolayevsky conjectured that this was the maximum number of edges on a thrackle on $S_g$:}

\begin{conjecture}[{\hglue -0.01 cm\cite[Conjecture 2]{cairnsbounds}}]\label{con:conj1}
If a graph $G$ with $n$ vertices and $m$ edges can be thrackled on $S_g$, {then} $m \le n + 2g$.
\end{conjecture}

Conjecture~\ref{con:conj1} was proved in~\cite{cairnsbounds} for all graphs with five or fewer vertices, with the possible exception of $K_5$. Aiming to offer further support for Conjecture~\ref{con:conj1}, Cairns, McIntyre, and Nikolayevsky~\cite{cairnsk5} claimed to show that $K_{3,3}$ cannot be thrackled on the torus $S_1$. {However, this statement is incorrect}: as we illustrate in Figure~\ref{fig:130}, it is possible to thrackle $K_{3,3}$ on $S_1$.

% *********************************************************
\def\tf#1{{\Scale[2.8]{#1}}}
\begin{figure}[ht!]
\centering
\def\tf#1{{\Scale[2.2]{#1}}}
\def\tg#1{{\Scale[2.5]{#1}}}
\def\somea{{\Scale[2.5]{\text{\rm (a)}}}}
\def\someb{{\Scale[2.5]{\text{\rm (b)}}}}
\scalebox{0.28}{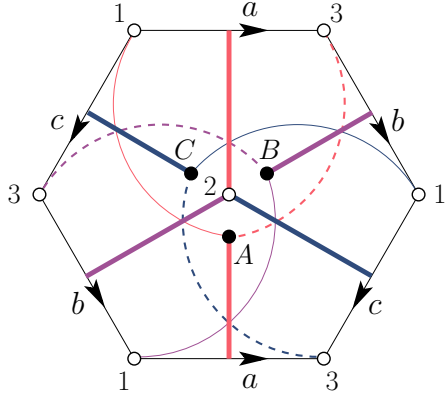}
\caption{A thrackle of $K_{3,3}$ on the torus, represented as a hexagon with its three pairs of opposite sides identified. The chromatic classes of $K_{3,3}$ are $\{A,B,C\}$ and $\{1,2,3\}$.}
\label{fig:130}
\end{figure}
% *********************************************************

\subsection{Our results}

We establish the following theorem, which in particular disproves Conjecture~\ref{con:conj1} for every $g>0$. % We emphasize that this statement involves bipartite graphs, which play an essential role in the investigation of thrackles on orientable surfaces~\cite{cairnsbounds,cairnsgeneralized,pelsmajereven}.

\begin{theorem}\label{thm:mainor}
{For each integer $g\ge 1$ there exist arbitrarily large connected graphs with $n$ vertices and $m=2n + 2g - 8$ edges that can be thrackled on $S_g$.}
\end{theorem}

This theorem reveals a stark contrast with the planar case. Indeed, as we mentioned above, if a graph with $n$ vertices and $m$ edges can be thrackled on the plane then $m \le 1.393n$. 

{Theorem~\ref{thm:mainor} easily implies a similar result for nonorientable surfaces, using that $S_h \# N_1 = N_{2h+1}$ for each integer $h \ge 0$. Indeed, a thrackle $T$ of a graph $G$ on $S_h$ trivially yields a thrackle of $G$ on $S_h\# N_1=N_{2h+1}$: if we glue to $S_h$ a (superfluous) crosscap disjoint from $T$, then $T$ becomes a thrackle of $G$ on $N_{2h+1}$.}

{Using this idea we obtain for each odd integer $g \ge 3$ the existence of a graph with $n$ vertices and $m=2n + g - 9$ edges that can be thrackled on $N_g$. Moreover, by gluing a second superfluous crosscap we obtain for each even integer $g \ge 4$ the existence of a graph with $n$ vertices and $m=2n + g - 10$ edges that can be thrackled on $N_g$.}

{We shall exhibit thrackles on nonorientable surfaces that not only have (marginally) more edges than the thrackles in the previous paragraph but, more importantly, are more natural constructions that do not involve the addition of superfluous crosscaps. We prove the following.}

\begin{theorem}\label{thm:mainnon}
{For each integer $g\ge 1$ there exist arbitrarily large connected graphs with $n$ vertices and $m=2n + g - 4$ edges that can be thrackled on $N_g$.}
\end{theorem}

We prove Theorem~\ref{thm:mainnon} first, and then we prove Theorem~\ref{thm:mainor}. We proceed in this order for two reasons: the proof of Theorem~\ref{thm:mainnon} is slightly easier, and this is the order in which we came up with these constructions.

In view of the next statement, Theorems~\ref{thm:mainor} and~\ref{thm:mainnon} seem to leave relatively little room for further progress related to thrackles on (orientable or nonorientable) nonplanar surfaces. %

\begin{theorem}\label{thm:upgendisc1}
Let $g>0$ be an integer. If a connected graph $G$ with $n$ vertices and $m$ edges can be thrackled on $S_g$, then $m \le 2n + 4g - 2$. If $G$ can be thrackled on $N_g$, then $m \le 2n + 2g - 2$. 
\end{theorem}

The proof of Theorem~\ref{thm:upgendisc1} is given in Section~\ref{sec:upper}. As we shall see, the part of the theorem for orientable surfaces was proved by Cairns and Nikolayevsky (see~\cite[Corollary 1]{cairnsgeneralized}). 

The part that concerns nonorientable surfaces {will follow by combining an argument developed by Cairns and Nikolayevsky~\cite{cairnsbounds} and a result proved by Pelsmajer, Schaefer, and \v{S}tefankovi\v{c}~\cite{pelsmajereven}}. As it happens, both the orientable and nonorientable parts will follow as consequences of a single statement (namely {Theorem~\ref{thm:upgendisc}}) that involves, as in~\cite{cairnsbounds}, thrackles of graphs that are not necessarily connected. Moreover, as in~\cite{cairnsbounds} and~\cite{cairnsgeneralized}, this result applies in the wider context of {\em generalized} thrackles, a notion we shall review at the beginning of Section~\ref{sec:upper}.

{We conclude this paper with some results on the thrackle genus of complete and complete bipartite graphs. We recall that the {\em thrackle genus\,} $\Tg(G)$ of a graph $G$ is the smallest $g$ such that $G$ can be thrackled on $S_g$. As we note in Section~\ref{sec:final}, using results by Cairns and Nikolayevsky and a classical result by Ringel it is easy to derive nontrivial lower bounds on $\Tg(K_{m,n})$ and $\Tg(K_n)$. As we also prove in that section, a construction used in the proof of Theorem~\ref{thm:mainor} can be easily adapted to yield upper bounds that are roughly twice as large as these lower bounds.}

%=====================================================

\section{Proof of Theorem~\ref{thm:mainnon}}\label{sec:proofnonorientable}

Theorem~\ref{thm:mainnon} is an easy consequence of the next two lemmas. For the first one we recall that for each integer $k\ge 3$ the $k$-{\em wheel} $W_k$ is the graph that consists of a $k$-cycle $C_k$ plus a vertex $a$ joined to all the vertices in $C_k$. The edges in $C_k$ are the {\em rim} edges, and the edges incident with $a$ are the {\em spokes}.

\begin{lemma}\label{lem:no1}
For each odd integer $k \ge 3$ the wheel $W_k$ can be thrackled on $N_1$.
\end{lemma}

%Since $W_k$ has $k+1$ vertices and $2k$ edges, Lemma~\ref{lem:no1} implies the first part of Theorem~\ref{thm:mainnon}. Moreover, it is worth noting that combining Lemma~\ref{lem:no1} with Lemma~\ref{lem:or2} we obtain that for each integer $t \ge 1$ is a graph with $n:=(k+1)+1$ vertices and $m:=2k + 2t+1$ edges that can be thrackled on $N_1 \# S_t$, which is homeomorphic to $N_{2t+1}$. Letting $g:=2t+1$, this means that for each odd integer $g \ge 3$ there is a graph with $n$ vertices and $m=2n+g-4$ edges that can be thrackled on $N_g$, which is exactly the second part of Theorem~\ref{thm:mainnon} for all odd integers $g \ge 3$.

%one obtains a proof of Theorem~\ref{thm:mainnon} for all odd integers $g \ge 3$. 

\begin{lemma}[{Cloning a star, nonorientable version}]\label{lem:no2}
{Let $G$ be a graph that can be thrackled on some surface $M$, and let $t$ be a positive integer. If $G$ has a vertex of degree at least $t+1$, then it is possible to add one vertex and $t+1$ edges to $G$ so that the resulting graph can be thrackled on $M\# N_{t}$.}
\end{lemma}

Before we prove these lemmas, we show that Theorem~\ref{thm:mainnon} is an easy consequence of them.

\begin{proof}[Proof of Theorem~\ref{thm:mainnon}, assuming Lemmas~\ref{lem:no1} and~\ref{lem:no2}]
{For $g=1$ the theorem claims the existence of arbitrarily large connected graphs with $n$ vertices and $m=2n+1 - 4=2n - 3$ edges that can be thrackled on $N_1$. For this case we note that Lemma~\ref{lem:no1} yields a slightly stronger statement, since $W_k$ has $k+1$ vertices and $2k$ edges, and $2k=2(k+1)-2$.}

Suppose now that $g \ge 2$ {is} an integer. Let $t:=g-1$, and let $k$ be any integer such that $k \ge t+1$. We know from Lemma~\ref{lem:no1} that $W_k$ can be thrackled on $N_1$. Since $W_k$ has a vertex of degree $k\ge t+1$, and it has $k+1$ vertices and $2k$ edges, it follows from Lemma~\ref{lem:no2} that there is a graph with $n:=(k+1) + 1=k+2$ vertices and $m:=2k + t + 1=2k+g$ edges that can be thrackled on $N_1\# N_{t}=N_1\# N_{g-1} =N_g$. We complete the proof by noting that {$m=2n + g -4$}, and that $n$ can be made arbitrarily large by making $k$ sufficiently large.
\end{proof}

\begin{proof}[Proof of Lemma~\ref{lem:no1}]
In Figure~\ref{fig:700} we illustrate a thrackle of $W_7$ on the projective plane $N_1$. It is straightforward to see that this drawing can be generalized to yield a thrackle of $W_k$ on $N_1$ for any odd integer $k\ge 3$.
\end{proof}

% *********************************************************
\def\tf#1{{\Scale[2.8]{#1}}}
\begin{figure}[ht!]
\centering
\def\tc#1{{\Scale[3]{#1}}}
\def\tf#1{{\Scale[2.2]{#1}}}
\def\tg#1{{\Scale[3]{#1}}}
\def\somea{{\Scale[2.5]{\text{\rm (a)}}}}
\def\someb{{\Scale[2.5]{\text{\rm (b)}}}}
\scalebox{0.25}{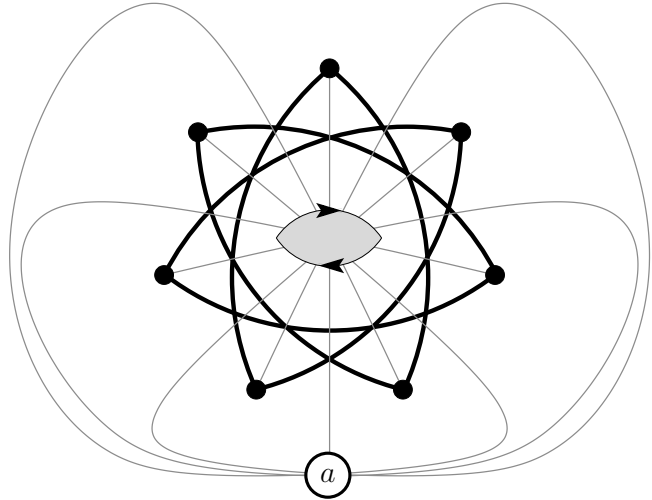}
\caption{A thrackle of the wheel $W_7$ on the projective plane $N_1$. The rim edges are thick and the spokes are thin, and the crosscap is the gray region shaped like an American football: the projective plane is obtained by identifying the upper and lower arcs that bound the gray region, using the orientations shown.} %, where the spokes  are drawn thin and the rim edges are drawn thick. To verify that this is indeed a thrackle it suffices to check that each spoke crosses all rim edges, except for the two rim edges adjacent to it.}
\label{fig:700}
\end{figure}

% *********************************************************

\begin{proof}[Proof of Lemma~\ref{lem:no2}]
{Throughout the proof $t$ is a positive integer and $G$ is a graph thrackled on some surface $M$, where $G$ has a vertex $u$ of degree $s \ge t+1$. Our discussion involves arbitrary values of $t$ and $s$, while we illustrate with figures the case in which $t=5$ and $s=9$.} 

As we illustrate in {Figure~\ref{fig:715}}, we label {by} $1,2,\ldots,s$ the vertices adjacent to $u$, and we let $e_i$ denote the edge that joins $u$ to $i$, for $i=1,2,\ldots,s$. We choose the labelling so that the edges $e_1,e_2,\ldots,e_s$ leave $u$ in this {clockwise} cyclic order.

% *********************************************************
\def\tf#1{{\Scale[2.8]{#1}}}
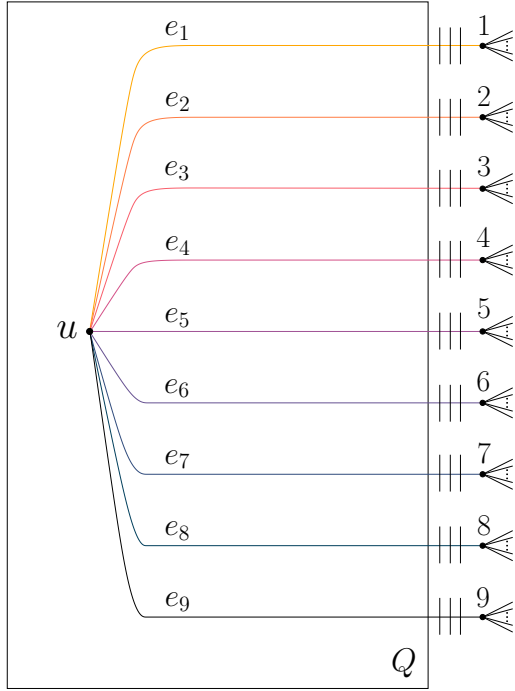
\begin{figure}[ht!]
\centering
\def\tx#1{{\Scale[14]{#1}}}
\def\utf#1{{\Scale[8]{#1}}}
\def\tf#1{{\Scale[11]{#1}}}
\def\tix#1{{\Scale[12]{#1}}}
\def\tif#1{{\Scale[11.5]{#1}}}
\hglue -0.2cm\scalebox{0.05}{\input{715c.pdf_t}}
\caption{Illustration of the proof of Lemma~\ref{lem:no2} for the case $s=9$.}
\label{fig:715}
\end{figure}
% *********************************************************

{To prove the theorem we shall show that it is possible to add to $G$ one new vertex $v$ and $t+1$ new edges $f_1,f_2,\ldots,f_{t}, f_s$, so that the resulting graph can be thrackled on $M\# N_t$.}

Without loss of generality we may assume that the layout is as depicted in {Figure~\ref{fig:715}}, where $Q$ is the {rectangle}. In particular, as we also illustrate in that figure, all the crossings of $e_i$ for $i=1,2,\ldots,s$ lie outside $Q$ (in the figure, these crossings are outlined as small vertical segments very close to vertices $1,2,\ldots,s$).

We now proceed as illustrated in Figure~\ref{fig:745}. We start by gluing $t$ crosscaps $\XX_1,\XX_2,\ldots,\XX_t$ inside the {rectangle} $Q$, redrawing $e_i$ so that it goes through $\XX_i$ for $i=1,2,\ldots,t$. 

% *********************************************************
\def\tf#1{{\Scale[2.8]{#1}}}
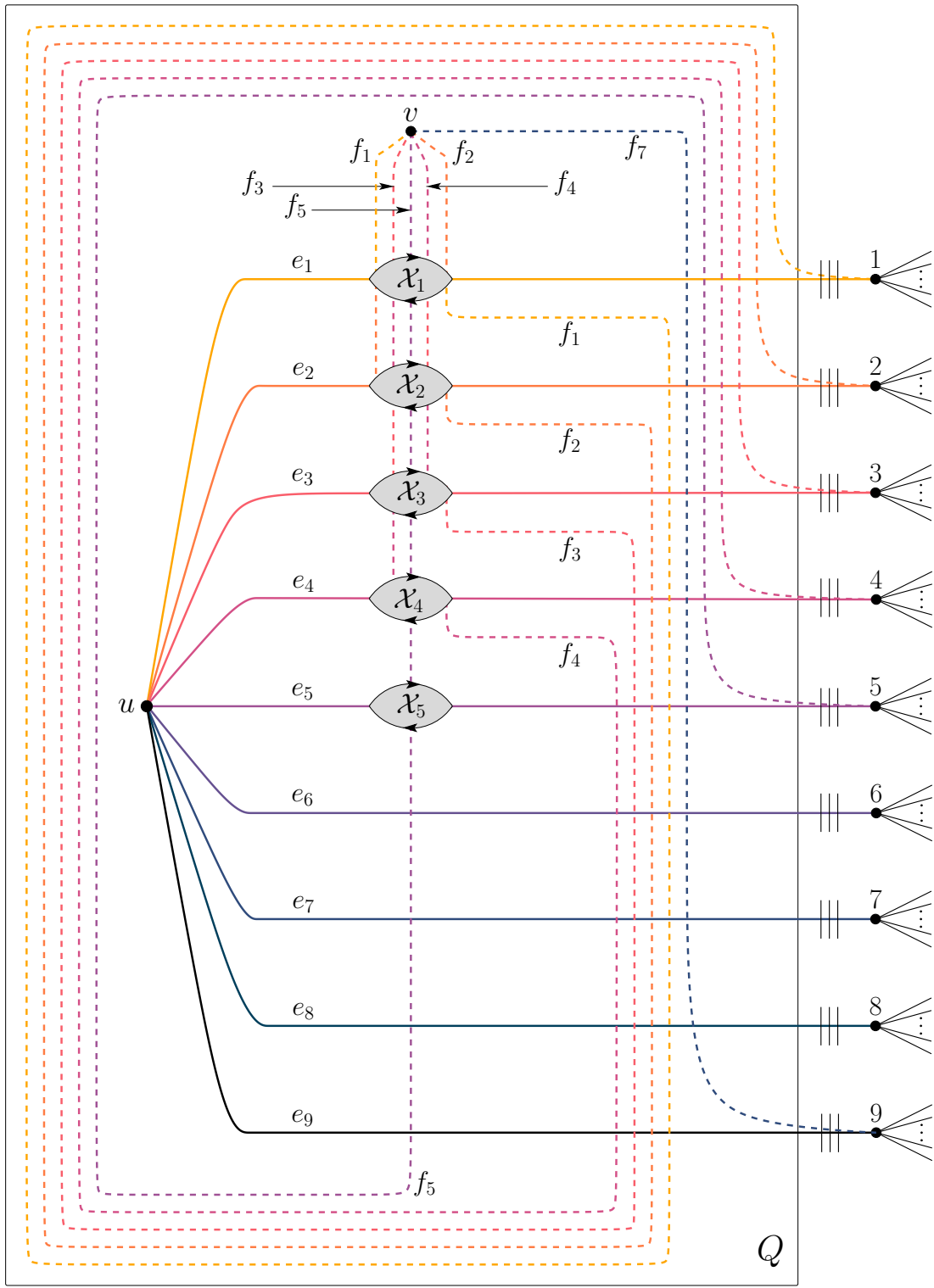
\begin{figure}[ht!]
\centering
\def\hone{{{\mathcal H}_1}}
\def\htwo{{{\mathcal H}_2}}
\def\hthree{{{\mathcal H}_3}}
\def\tc#1{{\Scale[5.5]{#1}}}
\def\tx#1{{\Scale[6]{#1}}}
\def\utf#1{{\Scale[8]{#1}}}
\def\tf#1{{\Scale[5.4]{#1}}}
\def\tif#1{{\Scale[4.0]{#1}}}
\def\tig#1{{\Scale[7.0]{#1}}}
\def\tix#1{{\Scale[7.0]{#1}}}
\def\tg#1{{\Scale[3]{#1}}}
\hglue -0.2cm\scalebox{0.1095}{\input{745j.pdf_t}}
%\hglue -0.2cm\scalebox{0.1095}{\input{745g.pdf_t}}
\caption{Cloning a star (nonorientable version).}
\label{fig:745}
\end{figure}
% *********************************************************

We then place a new vertex $v$ also inside $Q$, and for $i=1,2,\ldots,t$ and $s$ we join $v$ to vertex $i$ with an edge $f_i$. These are the dotted edges in Figure~\ref{fig:745}, and to help comprehension for $i=1,2,\ldots,t$ and $s$ the edge $f_i$ is drawn with the same colour as $e_i$.  If $t$ is odd (as in Figure~\ref{fig:745}), these edges leave $v$ in the {clockwise} cyclic order $f_2, f_4, \ldots, f_{t-3}, f_{t-1}, f_t, f_{t-2}, \ldots, f_3, f_1, f_s$. If $t$ is even, they leave $v$ in the cyclic order {$f_2, f_4, \ldots, f_{t-2}, f_t, f_{t-1}, f_{t-3}, \ldots, f_3, f_1, f_s$}.

As we illustrate in Figure~\ref{fig:745}, the key idea is that for $i=1,2,\ldots,t$, after leaving $v$ the edge $f_i$ goes through crosscaps $\XX_1,\XX_2,\ldots,\XX_i$ in this order, then crosses the edges $e_{i+1}, e_{i+2}, \ldots, e_s,e_1,e_2,\ldots,$ $e_{i-1}$ in this order, and finally is routed very close to $e_i$, so that just before it reaches vertex $i$ it crosses all the edges that cross $e_i$. It is straightforward to check that by drawing $f_i$ in this way we ensure that it crosses exactly once each edge incident with neither $v$ nor $i$, in agreement with the thrackle property.

{We finish the process by drawing $f_s$ so that it does not go through any of the crosscaps, crosses $e_1,e_2,\ldots,e_{s-1}$ in this order, and finally is routed very close to $e_s$, so that just before it reaches vertex $s$ it crosses all the edges that cross $e_s$. This guarantees that $f_s$ crosses exactly once each edge incident with neither $v$ nor $s$, in agreement with the thrackle property.}

{The final result of this construction is a thrackle on $M\# N_t$ of a graph that is obtained from $G$ by adding one vertex and $t+1$ edges, as required.}
\end{proof}

\section{Proof of Theorem~\ref{thm:mainor}}\label{sec:prooforientable}

Theorem~\ref{thm:mainor} follows from the next two statements. The first one involves the following family of graphs. For each integer {$k\ge 1$} let $G_k$ be the graph with vertex set $\{a,b,1,2,\ldots,2k\}$, where vertex $i$ is joined to vertex $i+1$ for $i=1,\ldots,2k-1$ (so that $(1,2,\ldots,2k)$ is a path), $a$ is joined to $1,3,5,\ldots,2k-1$ and $b$ is joined to $2,4,6,\ldots,2k$. In Figure~\ref{fig:665} we illustrate $G_5$. 

%%%

% 665

% The graph G_k (bipartite suspension of a path)

%%%

% *********************************************************
\def\tf#1{{\Scale[2.8]{#1}}}
\begin{figure}[ht!]
\centering
\def\tc#1{{\Scale[5.4]{#1}}}
\def\tf#1{{\Scale[2.4]{#1}}}
\def\tg#1{{\Scale[3]{#1}}}
\def\th#1{{\Scale[4]{#1}}}
\def\somea{{\Scale[3.5]{\text{\rm (a)}}}}
\def\someb{{\Scale[2.5]{\text{\rm (b)}}}}
\scalebox{0.22}{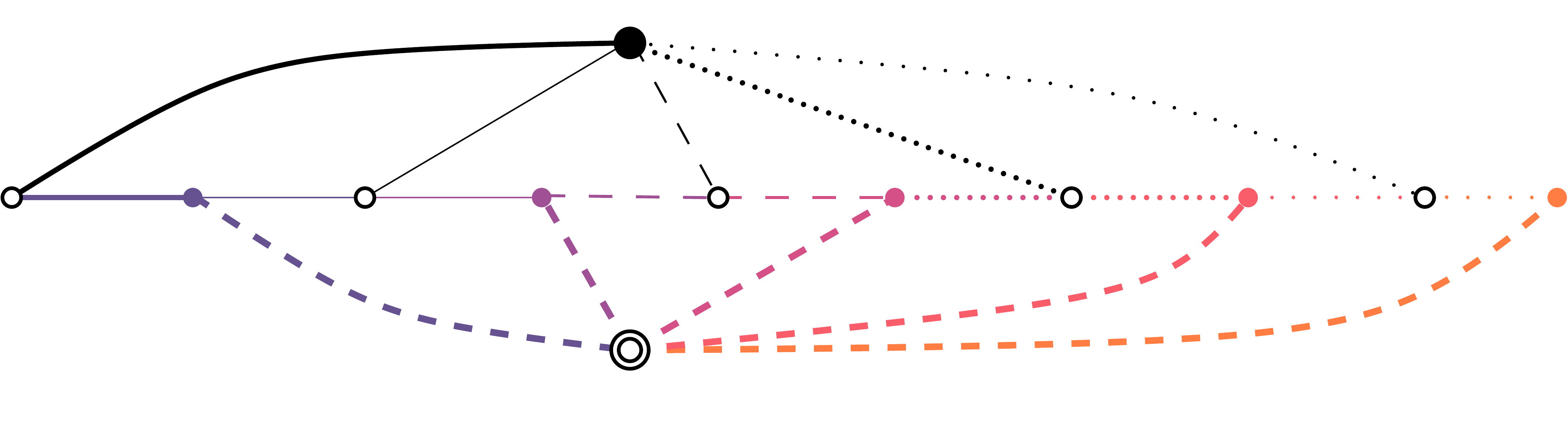}
\caption{The graph $G_5$ is bipartite, with bipartition classes $A=\{a,2,4,6,8,10\}$ and $B=\{b,1,3,5,7,9\}$. Each edge is drawn with one of six colours, and is of one of six {\em types}: thin solid, thick solid, thin dashed, big dotted, small dotted, or thick dashed. For each vertex $v$ in $A$, the edges incident with $v$ are of the same colour, and for each $v$ in $B$ the edges incident with $v$ are of the same type. Therefore in order to verify that a drawing of this graph is a thrackle, one must check that two edges cross each other (exactly once) if and only if they have different colours and are of different types.
}
\label{fig:665}
\end{figure}
% *********************************************************

\begin{lemma}\label{lem:or1}
For each integer {$k \ge 1$} the graph $G_k$ can be thrackled on $S_1$.
\end{lemma}

\begin{lemma}[{Cloning a star, orientable version}]\label{lem:or2}
{Let $G$ be a graph that can be thrackled on some surface $M$, and let $t$ be a positive integer. If $G$ has a vertex of degree at least $2t+1$, then it is possible to add one vertex and $2t+1$ edges to $G$ so that the resulting graph can be thrackled on $M\# S_{t}$.}
\end{lemma}

Before we prove these lemmas, we show that Theorem~\ref{thm:mainor} is an easy consequence of them.

\begin{proof}[Proof of Theorem~\ref{thm:mainor}, assuming Lemmas~\ref{lem:or1} and~\ref{lem:or2}]
{For $g=1$ the theorem claims the existence of arbitrarily large connected graphs with $n$ vertices and $m=2n+2{\cdot}1 - 8=2n - 6$ edges that can be thrackled on $S_1$. For this case we note that Lemma~\ref{lem:or1} actually yields a slightly stronger statement, since $G_k$ has $2k+2$ vertices and $4k-1$ edges, and $4k-1=2(2k+2)-5$.}

{Suppose now that $g \ge 2$. Let $t:=g-1$, and let $k$ be any integer such that $k \ge 2t+1$. We know from Lemma~\ref{lem:or1} that the graph $G_k$ can be thrackled on $S_1$. Since $G_k$ has a vertex of degree $k\ge 2t+1$, and it has $2k+2$ vertices and $4k-1$ edges, it follows from Lemma~\ref{lem:or2} that there is a graph with $n:=(2k+2) + 1=2k+3$ vertices and $m:=(4k-1) + 2t + 1=4k+2g-2$ edges that can be thrackled on $S_1\# S_{t}=S_1\# S_{g-1} =S_g$. We complete the proof by noting that {$m=2n+2g -8$}, and that $n$ can be made arbitrarily large by making $k$ sufficiently large}.
\end{proof}

\begin{proof}[Proof of Lemma~\ref{lem:or1}]
In Figure~\ref{fig:670} we illustrate a thrackle of $G_5$ on the torus $S_1$. The torus is obtained by removing the gray open disks and identifying their boundary circles using the orientations shown. It is straightforward to generalize this drawing to yield a thrackle of $G_k$ on $S_1$ for any integer ${k\ge 1}$.
\end{proof}

%%

% 670

% G_k  thrackled on the torus S_1

%%

% *********************************************************
\def\tf#1{{\Scale[2.8]{#1}}}
\begin{figure}[ht!]
\centering
\def\tc#1{{\Scale[5.4]{#1}}}
\def\tf#1{{\Scale[2.6]{#1}}}
\def\tg#1{{\Scale[3]{#1}}}
\def\th#1{{\Scale[4]{#1}}}
\def\somea{{\Scale[3.5]{\text{\rm (a)}}}}
\def\someb{{\Scale[2.5]{\text{\rm (b)}}}}
\scalebox{0.2}{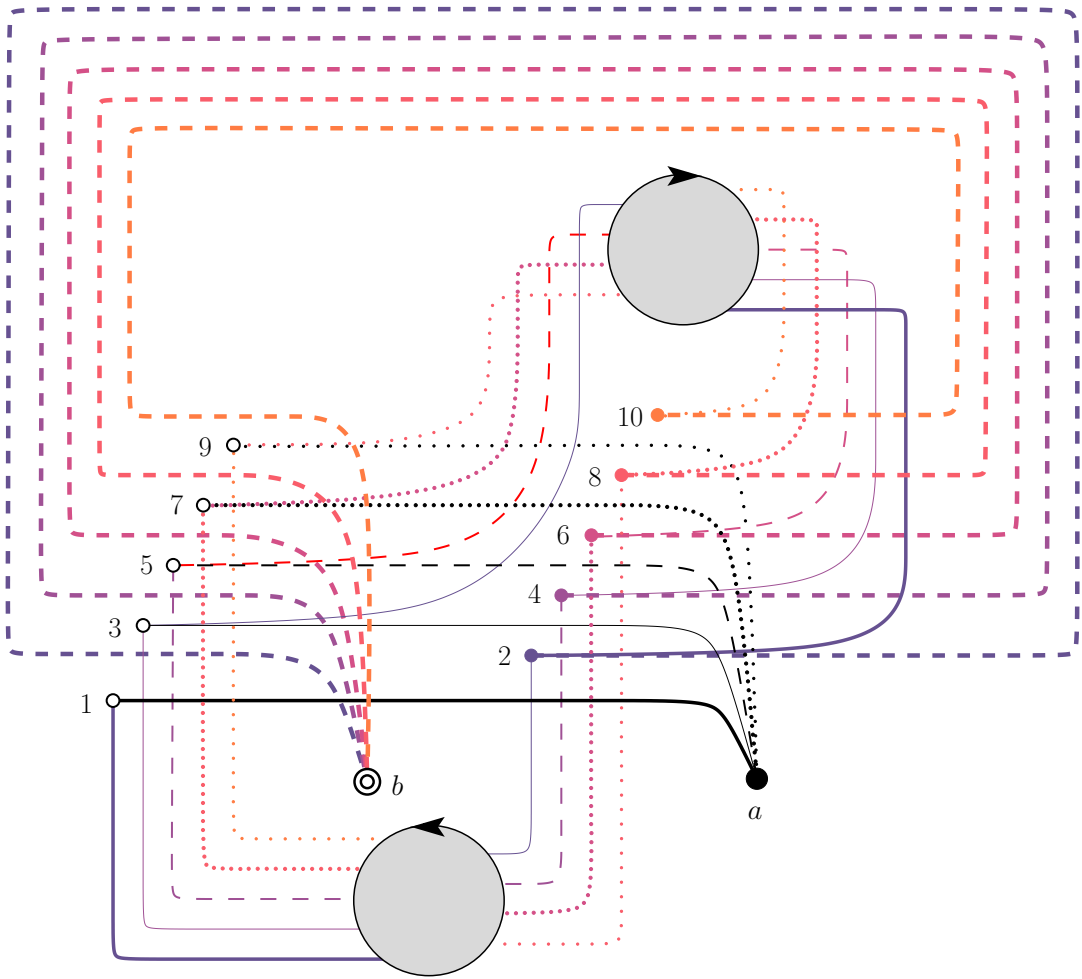}
\caption{A thrackle of $G_5$ on $S_1$. Each edge has the same colour and type as in Figure~\ref{fig:665}. Therefore, as noted in the caption of that figure, in order to verify that this is indeed a thrackle of $G_5$ one must check that two edges cross each other (exactly once) if and only if they have different colours and are of different type.}
\label{fig:670}
\end{figure}
% *********************************************************

\begin{proof}[Proof of Lemma~\ref{lem:or2}]
{Throughout the proof $t$ is a positive integer and $G$ is a graph thrackled on some surface $M$, where $G$ has a vertex $u$ of degree $s \ge 2t+1$. Our discussion involves arbitrary values of $t$ and $s$, while we illustrate with figures the case in which $t=3$ and $s=9$}. 

{As in the proof of Lemma~\ref{lem:no2} (see Figure~\ref{fig:715}), we label by $1,2,\ldots,s$ the vertices adjacent to $u$, and we let $e_i$ denote the edge that joins $u$ to $i$, for $i=1,2,\ldots,s$. We choose the labelling so that the edges $e_1,e_2,\ldots,e_s$ leave $u$ in this cyclic order.}

%%

% 715

% The neighbourhood of u, a degree 9 vertex

%%

{To prove the theorem we shall show that it is possible to add to $G$ one new vertex $v$ and $2t+1$ new edges $f_1,f_2,\ldots,f_{2t-1},f_{2t}, f_s$, so that the resulting graph can be thrackled on $M\# S_t$.}

We let $Q$ be a {rectangle} that contains $u$ and a (connected) part of $e_i$ for $i=1,2,\ldots,s$, and does not intersect any other vertex or edge in the drawing. As in the proof of Lemma~\ref{lem:no2}, without loss of generality we may assume that the layout is as depicted in Figure~\ref{fig:715}.

We now proceed as illustrated in Figure~\ref{fig:740}. We start by gluing $t$ handles {$\HH_1,\HH_2,\ldots,\HH_t$} inside the {rectangle} $Q$. We use these handles to redraw edges $e_1,e_2,\ldots,e_{2t}$, so that edges $e_i$ and $e_{i+1}$ go ``through'' handle $\HH_{(i+1)/2}$ for $i=1,3,5,\ldots,2t-1$. Note that the edges incident with $u$ now leave $u$ in the {clockwise} cyclic order $e_2,e_1,e_4,e_3,e_6,e_5,\ldots,e_{2t},e_{2t-1},e_{2t+1},e_{2t+2},e_{2t+3},\ldots,e_{s-1},e_s$.

% *********************************************************
\def\tf#1{{\Scale[2.8]{#1}}}
\begin{figure}[ht!]
\centering
\def\hone{{{\mathcal H}_1}}
\def\htwo{{{\mathcal H}_2}}
\def\hthree{{{\mathcal H}_3}}
\def\tc#1{{\Scale[5.5]{#1}}}
\def\tx#1{{\Scale[6]{#1}}}
\def\utf#1{{\Scale[8]{#1}}}
\def\tf#1{{\Scale[4.5]{#1}}}
\def\tg#1{{\Scale[3]{#1}}}
\def\tix#1{{\Scale[5.5]{#1}}}
\def\somea{{\Scale[2.5]{\text{\rm (a)}}}}
\def\someb{{\Scale[2.5]{\text{\rm (b)}}}}
\hglue -0.2cm\scalebox{0.125}{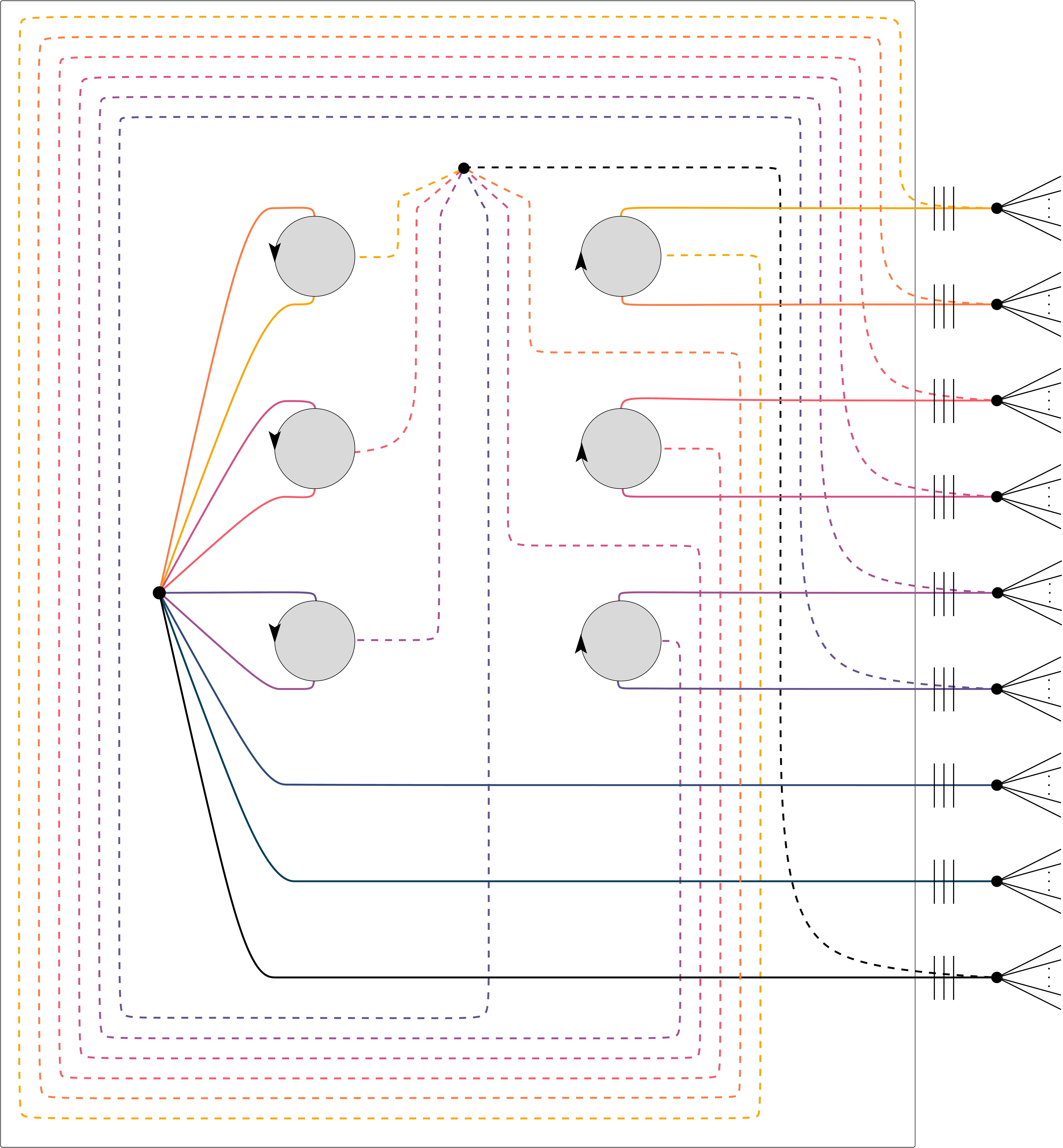}
\caption{Cloning a star (orientable version).}
\label{fig:740}
\end{figure}
% *********************************************************

Next we place a new vertex $v$ also inside $Q$, and for $i=1,2,\ldots,2t$ and $s$ we join $v$ to vertex $i$ with an edge $f_i$. These $2t+1$ new edges are the dotted edges in Figure~\ref{fig:740}, and to help comprehension for $i=1,2,\ldots,t$ and $s$ the edge $f_i$ is drawn with the same colour as $e_i$. These edges leave $v$ in the {clockwise} cyclic order {$f_2, f_4, \ldots, f_{2t-4}, f_{2t-2}, f_{2t}, f_{2t-1}, \ldots, f_3, f_1,f_s$}.

As we also illustrate in {Figure~\ref{fig:740}}, for $i=1,3,5,\ldots,2t-1$ we draw $f_i$ so that after leaving $v$ it goes through handle $\HH_{(i+1)/2}$ (``sharing'' the handle with $e_i$ and $e_{i+1}$), then crosses the edges $e_{i+1},e_{i+2},\ldots,e_{s},e_1,e_2,\ldots,e_{i-1}$ in this order, and finally is routed very close to $e_{i}$, so that just before it reaches vertex $i$ it crosses (outside $Q$) all the edges that cross $e_i$. It is straightforward to check that by drawing $f_i$ in this way we ensure that it crosses exactly once each edge incident with neither $v$ nor $i$, in agreement with the thrackle property.

Now for $i=2,4,6,\ldots,2t$ we draw the edge $f_i$ as follows. After leaving $v$ the edge $f_i$ goes ``underneath'' handles $\HH_{1},\HH_2,\ldots,\HH_{i/2}$, and then (as in the case in which $i$ is odd) it crosses the edges $e_{i+1},e_{i+2},\ldots,e_{s},e_1,e_2,\ldots,e_{i-1}$ in this order, and finally is routed very close to $e_{i}$, so that just before it reaches vertex $i$ it crosses all the edges that cross $e_i$. As in the case in which $i$ is odd, it is straightforward to check that by drawing $f_i$ in this way we ensure that it crosses exactly once each edge incident with neither $v$ nor $i$, in agreement with the thrackle property.

Finally, as we also illustrate in Figure~\ref{fig:740}, we draw $f_s$ in an exceptional way, so that it crosses $e_1, e_2, \ldots, e_{s-1}$ in this order, and finally route it very close to $e_s$, so that before it reaches vertex $s$ it crosses all the edges that cross $e_s$. As in the previous cases, this guarantees that $f_s$ crosses exactly once each edge incident with neither $v$ nor $s$, in agreement with the thrackle property.

{The final result of this construction is a thrackle on $M\# S_t$ of a graph that is obtained from $G$ by adding one vertex and $2t+1$ edges, as required.}
\end{proof}

%

%=====================================================
\section{Proof of Theorem~\ref{thm:upgendisc1} and more general upper bounds}\label{sec:upper}

%As we shall see shortly, 

%As it happens, both the orientable and nonorientable part will follow as consequences of a single statement (namely Theorem~\ref{thm:upgendisc1}) that involves, as in~\cite{cairnsbounds}, thrackles of graphs that are not necessarily connected. Moreover, as in~\cite{cairnsbounds} and~\cite{cairnsgeneralized}, this result applies in the wider context of {\em generalized} thrackles, a notion we review at the beginning of Section~\ref{sec:upper}.

We prove Theorem~\ref{thm:upgendisc1} as part of a more general discussion on general upper bounds on the number of edges of a graph that can be thrackled on an orientable or nonorientable surface of positive genus. In particular, Theorem~\ref{thm:upgendisc1} will follow as a particular case of Theorem~\ref{thm:upgendisc} below.

%We remark that the part of Theorem~\ref{thm:upgendisc1} for orientable surfaces was proved by Cairns and Nikolayevsky (see~\cite[Corollary 1]{cairnsgeneralized}). The part that concerns nonorientable surfaces will follow by combining an argument in~\cite{cairnsbounds} and a result from~\cite{pelsmajereven}.

We note that the best available bounds, as well as the bounds we prove below, actually hold in the broader context of generalized thrackles. We recall that a {\em generalized thrackle} is a drawing of a graph on some surface such that (i) each pair of adjacent edges intersect each other (other than at their common endvertex) at an even number of points (possibly zero), at which they cross; and (ii) each pair of nonadjacent edges intersect each other in an odd number of points, at which they cross. Evidently, a thrackle is a particular instance of a generalized thrackle.

{Cairns and Nikolayevsky established the following general upper bound for bipartite graphs.}

\begin{theorem}[\hglue -0.01 cm{\cite[Theorem 4(a)]{cairnsbounds}}]\label{thm:upbip}
{Let $G$ be a bipartite graph with $n$ vertices and $m$ edges. Suppose that $G$ has $k$ connected components, each of which has at least three vertices, and let $g \ge 0$ be an integer. If there exists a generalized thrackle of $G$ on $S_g$, then $m \le 2n -4k + 4g$.}
\end{theorem}

The condition that $G$ has no connected components with fewer than three vertices was not stated in~\cite{cairnsbounds}, but this is an essential hypothesis. To see this, suppose that {$G$ is a nonnull graph with $s$ connected components that consist of a single vertex, $t$ connected components that have exactly two vertices (that is, they are isomorphic to $K_2$), and no connected components with three or more vertices}. Thus $n=s+2t$ and $m=t$. Since $G$ is bipartite and it can clearly be thrackled on $S_0$ (the sphere), the inequality in Theorem~\ref{thm:upbip} would imply that $m \le 2n - 4(s+t)$, that is, $t \le 2(s+2t) -4(s+t)$, and so $t \le -2s$. But this last inequality cannot hold, since $s$ and $t$ are both nonnegative integers and at least one of them is positive.

{In a subsequent paper, Cairns and Nikolayevsky proved the following general upper bound for arbitrary (that is, not necessarily bipartite) connected graphs.}

\begin{theorem}[\hglue -0.01 cm{\cite[Corollary 1]{cairnsgeneralized}}]\label{thm:upgen}
{Let $G$ be a connected graph with $n$ vertices and $m$ edges, and let $g \ge 0$ be an integer. If there exists a generalized thrackle of $G$ on $S_g$, then $m \le 2n - 2 + 4g$.}
\end{theorem}

Our aim in this final section is to note that, {using arguments by Cairns and Nikolayevsky~\cite{cairnsbounds} and a characterization of generalized thrackles on surfaces by Pelsmajer, Schaefer, and \v{S}tefankovi\v{c}~\cite{pelsmajereven}, Theorems~\ref{thm:upbip} and~\ref{thm:upgen} can be easily extended to generalized thrackles of arbitrary (connected or not) graphs on any surface, orientable or nonorientable}. This is the content of the next statement. We recall that the {\em Euler genus} $\varepsilon(M)$ of a surface $M$ of genus $g$ is $2g$ if $M$ is orientable, and $g$ if $M$ is nonorientable.

\begin{theorem}[Implies Theorem~\ref{thm:upgendisc1}]\label{thm:upgendisc}
{Let $G$ be a graph with $n$ vertices and $m$ edges. Suppose that $G$ has $k$ connected components, each of which has at least three vertices, and let $M$ be a surface (orientable or nonorientable). If there exists a generalized thrackle of $G$ on $M$, then $m \le 2n -4k + 2 \varepsilon(M) + 2$.}
\end{theorem}

{For completeness, at the end of this paper (see Corollary~\ref{cor:upgendisc}) we establish the corresponding version of this theorem in which components with fewer than three vertices are allowed.}

{Recall that a graph embedding is~\emph{cellular} if every face is homeomorphic to an open disk. We say that an embedding (cellular or not) is {\em even} if every facial boundary walk has even length.}

{We use the following result by Pelsmajer, Schaefer, and \v{S}tefankovi\v{c}~\cite{pelsmajereven}. If $N$ is a nonorientable surface with a specified crosscap $X$, an {\em $X$-parity embedding} is an embedding of a graph on $N$ in which a cycle is odd if and only if it passes through $X$ an odd number of times. We make essential use of the observation that every $X$-parity embedding is even.} 

\begin{theorem}[\hglue -0.01 cm{\cite[Theorem 4.6]{pelsmajereven}}]
\label{theorem_X_parity}%{\rm{\cite[Theorem 4.6]{pelsmajereven}}}
A graph $G$ has a generalized thrackle on a surface $M$ if and only if $G$ has an $X$-parity embedding on $M\# N_1$.
\end{theorem}

\begin{proof}[Proof of Theorem~\ref{thm:upgendisc}]
Let $G$ be a graph with $n$ vertices and $m$ edges. Suppose that $G$ has $k$ connected components $G_1,\ldots,G_k$, each of which has at least three vertices. For $i=1,\ldots,k$ we let $n_i$ (respectively, $m_i$) denote the number of vertices (respectively, edges) in $G_i$.

Suppose that there exists a generalized thrackle of $G$ on a surface $M$. The ``only if'' part of Theorem~\ref{theorem_X_parity} then implies that $G$ has an $X$-parity embedding {on the surface obtained by adding a crosscap $X$ to $M$. Since every $X$-parity embedding is even, we conclude that $G$ has an even embedding on $M\# N_1$.} 

We now proceed as in the proof of~\cite[Theorem 4]{cairnsbounds}. Since $G$ has an even embedding on $M\# N_1$ it follows that for $i=1,\ldots,k$ there is a cellular even embedding $\Psi_i$ of $G_i$ on a surface $M_i$ where 
\begin{equation}
\sum_{i=1}^k \varepsilon(M_i) \le \varepsilon(M\# N_1)= \varepsilon(M) + 1.
\label{eq_euler}
\end{equation}

{Since none of $G_1,\ldots,G_k$ consists of a single isolated vertex (each connected component of $G$ has at least three vertices) it follows that each facial boundary walk of $\Psi_i$ has length at least $2$.}

Now a facial boundary walk in an embedding has length {exactly} $2$ only if the graph has parallel edges or has at least one component isomorphic to $K_2$. Since in our context all graphs under consideration are simple, and {no connected component of $G$ has fewer than three vertices}, we conclude that for $i=1,\ldots,k$, every facial boundary walk of $\Psi_i$ has length at least $4$.

Let $i\in \{1,\ldots,k\}$. If $f_i$ is the number of faces in $\Psi_i$, {since $\Psi_i$ is cellular it follows from} Euler's formula that $2-\varepsilon(M_i) = n_i - m_i + f_i$. Since each facial boundary walk of $\Psi_i$ has length at least $4$ it follows that $2m_i  \ge 4f_i$. Combining these two inequalities we obtain that $m_i \le 2n_i  -4 + 2\varepsilon(M_i)$. Therefore
\[ 
m = \sum_{i=1}^k m_i \le \sum_{i=1}^k \biggl(  2n_i  - 4 + 2\varepsilon(M_i)  \biggr) = 2n - 4k + 2\sum_{i=1}^k \varepsilon(M_i).
\] 

Using this inequality and~\eqref{eq_euler} we obtain $m \le 2n -4k + 2(\varepsilon(M)+1)$, that is, $m \le 2n -4k + 2\varepsilon(M) + 2$.
\end{proof}

{The discussion after Theorem~\ref{thm:upbip} implies that the condition in Theorem~\ref{thm:upgendisc} that no connected component has fewer than three vertices cannot be omitted. For completeness, we finally include the corresponding version of Theorem~\ref{thm:upgendisc} without restrictions on the sizes of the connected components of the graphs under consideration.}

\begin{corollary}\label{cor:upgendisc}
{Let $G$ be a graph with $n$ vertices and $m$ edges. Suppose that $G$ has $k$ connected components, $s$ of which have have exactly one vertex and $t$ of which have exactly two vertices. Further suppose that at least one component of $G$ has at least three vertices. If there exists a generalized thrackle of $G$ on a surface $M$, then $m \le 2n - 4k + 2 \varepsilon(M) + 2s +t + 2$.}
\end{corollary}

We remark that the condition in Corollary~\ref{cor:upgendisc} that there is at least one component with at least three vertices leaves out only the case in which each component has at most two vertices, but this case is completely understood, since such a graph is thrackleable on the sphere $S_0$, and hence on any surface. %. Indeed, if $G$ has $k$ components, $s$ of which have exactly one vertex, $t$ of which have exactly two vertices, and none of which has more than two vertices, then $G$ is clearly thrackleable on the sphere $S_0$ (and hence on any surface). Since $m=t, n=s+2t$, and $k=s+t$, in terms of $m,n$ and $k$ one simply has that $m=n-k$.

\begin{proof}[Proof of Corollary~\ref{cor:upgendisc}]
Let $G_1,\ldots,G_k$ be the connected components of $G$, labelled so that each of $G_1,\ldots,G_{k-(s+t)}$ has at least three vertices, each of $G_{k-(s+t)+1},\ldots, G_{k-t}$ has exactly one vertex, and each of $G_{k-t+1},\ldots,G_k$ has exactly two vertices. For $i=1,\ldots,k$ we let $n_i$ (respectively, $m_i$) denote the number of vertices (respectively, edges) in $G_i$. 

Note that for $i=k-(s+t)+1,\ldots,k-t$ we have $n_i=1$ and $m_i=0$, and for $i=k-t+1,\ldots,k$ we have $n_i=2$ and $m_i=1$. Therefore
\begin{equation}\label{eq:edges}
m=\sum_{i=1}^k m_i = \sum_{i=1}^{k-(s+t)} m_i + \sum_{i=k-(s+t)+1}^{k-t} m_i + \sum_{i=k-t+1}^{k} m_i = \sum_{i=1}^{k-(s+t)} m_i + t.
\end{equation}
Similarly, 
\[
n=\sum_{i=1}^k n_i = \sum_{i=1}^{k-(s+t)} n_i + \sum_{i=k-(s+t)+1}^{k-t} n_i + \sum_{i=k-t+1}^{k} n_i = \sum_{i=1}^{k-(s+t)} n_i + s + 2t,
\] 
and so
\begin{equation}\label{eq:vertices}
\sum_{i=1}^{k-(s+t)} n_i = n - s - 2t.
\end{equation}
Now the graph $G':=G_1\cup \cdots \cup G_{k-(s+t)}$ has $\sum_{i=1}^{k-(s+t)} n_i$ vertices, $\sum_{i=1}^{k-(s+t)} m_i$ edges, and $k-(s+t)$ connected components. Since each connected component of $G'$ has at least three vertices, Theorem~\ref{thm:upgendisc} implies that
%Now since each connected component of the graph $G_1\cup \cdots \cup G_{k-(s+t)}$ has at least three vertices, Theorem~\ref{thm:upgendisc} implies that
\begin{equation}\label{eq:disjun}
\sum_{i=1}^{k-(s+t)} m_i \le 2\biggl(\,\,\sum_{i=1}^{k-(s+t)} n_i \biggr) - 4(k-(s+t)) + 2\varepsilon(M) + 2. 
\end{equation}

Combining~\eqref{eq:edges} and~\eqref{eq:disjun} we obtain
\begin{equation*}\label{eq:ineq1}
m \le \biggl[ \biggl(\,\,2\sum_{i=1}^{k-(s+t)} n_i\biggr) - 4(k-(s+t)) + 2\varepsilon(M) + 2\biggr] + t.
\end{equation*}

\noindent In view of~\eqref{eq:vertices}, this last inequality implies that
\begin{align*}
m  &\le 2( n- s - 2t) - 4k + 4s + 4t + 2 \varepsilon(M) + 2 + t\\
& = 2n -4k + 2\varepsilon(M) + 2s + t + 2.\qedhere\\ 
\end{align*}\end{proof}

\section{{On the thrackle genus of $K_{m,n}$ and $K_n$}}\label{sec:final}

{Our aim in this section is to note that our constructions can be applied to obtain upper  bounds on the thrackle genus of complete and complete bipartite graphs. We recall that the {\em thrackle genus} $\Tg(G)$ of a graph $G$ is the smallest $g$ such that $G$ can be thrackled on $S_g$. This notion was introduced by Green and Ringeisen~\cite{greenringeisen}, who also proved that any given graph can be thrackled on some $S_g$.}

\subsection{Previous bounds on the thrackle genus}

{Given a graph $G$ let $\indep(G)$ denote the collection of all pairs of independent edges in $G$. Using this notation, a simple drawing of $G$ on some surface is a thrackle if and only if it has $|\indep(G)|$ crossings. Green and Ringeisen~\cite[Theorem 2]{greenringeisen} proved that for any integer $k$, $0\le k \le |\indep(G)|$, there is a simple drawing of $G$ on some $S_g$ with exactly $k$ crossings. As noted in~\cite[Corollary 3]{greenringeisen}, this implies in particular that every graph is thrackleable on some $S_g$.}

{Green and Ringeisen gave no explicit upper bounds on $\Tg(G)$ in that paper, but Cottingham and Ringeisen derived upper bounds in a subsequent paper~\cite{cottingham}. In a nutshell, Cottingham and Ringeisen noted that if $D$ is a simple drawing of a graph $G$ on some surface $S_h$, and $e,f$ are independent edges of $G$ that do not cross in $D$, then by adding at most two handles to $S_h$ it is possible to modify $D$ to obtain a drawing with the same crossings as $D$ plus a crossing between $e$ and $f$. An iterative application of this observation yields that if $G$ has a simple drawing on $S_h$ with $c$ crossings then $\Tg(G) \le h + 2 (|\indep(G)|-c)$~\cite[Theorem 5.1]{greenringeisen}. 
}

{
Using that $K_{m,n}$ has a simple drawing on $S_0$ with $\binom{m}{2} \binom{n}{2}$ crossings (place the vertices on the equator, with the vertices in each partite class appearing consecutively, and draw each edge on the Northern hemisphere), and that $\indep(K_{m,n}) = 2\binom{m}{2}\binom{n}{2}$, the inequality implies that $\Tg(K_{m,n}) \le 2\bigl( \binom{m}{2}\binom{n}{2} \bigr) = m^2n^2/2 + O(m^2n + mn^2)$. Similarly, using that $K_n$ has a simple drawing on $S_0$ with $\binom{n}{4}$ crossings (place the vertices on the equator, and draw each edge on the Northern hemisphere), and that $\indep(K_n) = 3\binom{n}{4}$, this inequality implies that $\Tg(K_n) \le 2(2\binom{n}{4}) = n^4/6 + O(n^3)$.
}

{
As it happens, the arguments used by Green and Ringeisen in the proof of~\cite[Theorem 2]{greenringeisen} can be used to derive slightly better bounds. In their proof they note that one can consider a drawing as an embedding, regarding crossings as degree four vertices. If the drawing of a graph $G=(V,E)$ under consideration is a thrackle, when regarded as an embedding it has $|V| + |\indep(G)|$ vertices, $|E| + 2|\indep(G)|$ edges, and (obviously) at least one face. Euler's formula then implies that the drawing can be hosted in $S_g$, where $2-2g \ge |V| + |\indep(G)| - (|E| + 2|\indep(G)|) + 1$, that is, where $g \le (1/2)(|E| - |V| + |\indep(G)| + 1)$. Therefore $\Tg(G) \le (1/2)(|E|-|V|+|\indep(G)|+1)$. 
}

%If $\indep'$ is any subset of $\indep(G)$, then there is a simple drawing $D$ of $G$ on $S_{(1/2)(|V(G)|-|E(G)|+|\indep'|+1)}$ such that the pairs of edges that cross in $D$ are precisely the pairs in $\indep'$. This implies in particular that there is a thrackle drawing of $G$ on $S_{(1/2)(|V(G)-|E(G)|+|\indep(G)|+1)}$. In other words, $\Tg(G) \le (1/2)(|V(G)|-|E(G)|+|\indep(G)|+1)$.

%{In a nutshell, the arguments in the proof of~\cite[Theorem 2]{greenringeisen} imply the following. Let $G=(V,E)$ be a graph, and let $\indep(G)$ denote the collection of all pairs of independent edges in $G$. If $\indep'$ is any subset of $\indep(G)$, then there is a simple drawing $D$ of $G$ on $S_{(1/2)(|V(G)|-|E(G)|+|\indep'|+1)}$ such that the pairs of edges that cross in $D$ are precisely the pairs in $\indep'$. This implies in particular that there is a thrackle drawing of $G$ on $S_{(1/2)(|V(G)-|E(G)|+|\indep(G)|+1)}$. In other words, $\Tg(G) \le (1/2)(|V(G)|-|E(G)|+|\indep(G)|+1)$.}

{
Since $K_{m,n}$ has $m+n$ vertices, $mn$ edges, and $|\indep(K_{m,n})|= 2\binom{m}{2}\binom{n}{2}$, this last inequality implies that $\Tg(K_{m,n}) \le m^2 n^2/4 + O(m^2 n + m n^2)$, and since $K_n$ has $n$ vertices, $\binom{n}{2}$ edges and $|\indep(K_n)|=3\binom{n}{4}$, it follows that $\Tg(K_n) \le n^4/16 + O(n^3)$. Thus indeed this last argument yields better upper bounds than~\cite[Theorem 5.1]{greenringeisen}.
}

{
In any case these upper bounds are far from optimal. Indeed, as we show below, $\Tg(K_{m,n}) =\Theta(mn)$ and and $\Tg(K_n) = \Theta(n^2)$. 
}

\subsection{{Upper bounds on $\Tg(K_{m,n})$ and $\Tg(K_n)$}}
Let us start by noting that the proof of Lemma~\ref{lem:or2} in particular shows the following. %, if we assume that $G$ has a vertex of degree {\em exactly} $2t+1$.

\begin{lemma}[Follows from the proof of Lemma~\ref{lem:or2}]\label{lem:or2p}
{Let $G$ be a graph that can be thrackled on some surface $M$. {Let $u$ be a vertex of $G$, and let $d$ be the degree of $u$.} If $G'$ is the graph obtained by adding a vertex $v$ and joining it to all the vertices adjacent to $u$, then $G'$ can be thrackled on {$M\# S_{\ceil{{(d-1)}/{2}}}$.}} 
\end{lemma}

{The lemma easily yields bounds on the thrackle genus of $K_{m,n}$, as follows.}

\begin{corollary}\label{cor:k2n}
{For any integer $n\ge 2$ we have}
\[
\Tg(K_{2,n}) \le \ceil{(n-1)/2}.
\]
\end{corollary}

\begin{proof}
{Since $K_{1,n}$ is trivially thrackled in the sphere $S_0$ {for any positive integer $n$, Lemma~\ref{lem:or2p} implies} that $K_{2,n}$ can be thrackled in $S_{\ceil{(n-1)/2}}$.}
\end{proof}

\begin{corollary}\label{cor:kMn}
{For any integers $m\ge 3$ and $n\ge 3$ we have}
\[
{\Tg(K_{m,n})}\le 
\begin{cases}
{\dfrac{(m-1)(n-1)}{2}} - 1, &\mbox{{if at least one of $m$ and $n$ is odd}} \\[0.5cm]
{\dfrac{(m-1)(n-1)-1}{2}}, &\mbox{{if both $m$ and $n$ are even.}}
\end{cases}
\]
\end{corollary}

\begin{proof}
{Suppose that at least one of $m$ and $n$ (say $n$, without loss of generality) is odd.} Since $K_{3,3}$ can be thrackled on $S_1$ (see Figure~\ref{fig:130}), applying Lemma~\ref{lem:or2p} iteratively $n-3$ times {we obtain a thrackle drawing of $K_{3,n}$ on $S_{1 + (n-3)(3-1)/2}=S_{n-2}$}. Since $n$ is odd we can now apply the lemma iteratively $m-3$ times, obtaining that $K_{m,n}$ can be thrackled on $S_{n-2+(m-3)(n-1)/2}=S_{(m-1)(n-1)/2{-1}}$. Thus the corollary holds if at least one of $m$ and $n$ is odd.

{Suppose finally that both $m$ and $n$ are even.} {By assumption $m\ge 3$, and so $m\ge 4$. Since $m-1$ is at least $3$ and it is odd, then the corollary holds for $K_{m-1,n}$. Thus $K_{m-1,n}$ can be thrackled on $M=S_{(m-2)(n-1)/2 -1}$. An application of Lemma~\ref{lem:or2p} then implies that $K_{m,n}$ can be thrackled on $S_{((m-2)(n-1)/2 -1) + n/2} = S_{((m-1)(n-1)-1)/2}$.}
\end{proof}

Regarding complete graphs we start by noting that $K_1, K_2$, and $K_3$ are trivially thrackleable on $S_0$. {A straightforward case analysis shows that} $C_4$ cannot be thrackled on $S_0$, and consequently $K_4$ cannot be thrackled on $S_0$. {On the other hand, $K_4$ can be thrackled on $S_1$, as shown by Cottingham and Ringeisen~\cite[Figure 3]{cottingham}. In Figure~\ref{fig:165} we illustrate an alternative, more symmetrical representation of this thrackle.}

% *********************************************************
\def\tf#1{{\Scale[2.8]{#1}}}
\begin{figure}[ht!]
\centering
\def\tf#1{{\Scale[2.2]{#1}}}
\def\tg#1{{\Scale[2.5]{#1}}}
\def\somea{{\Scale[2.5]{\text{\rm (a)}}}}
\def\someb{{\Scale[2.5]{\text{\rm (b)}}}}
\scalebox{0.28}{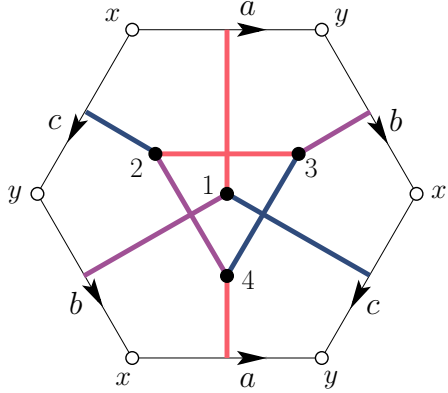}
\caption{A thrackle of $K_{4}$ on the torus, represented as a hexagon with its three pairs of opposite sides identified.}
\label{fig:165}
\end{figure}
% *********************************************************

{To obtain bounds on $\Tg(K_n)$ for $n > 4$} we note that the proof of Lemma~\ref{lem:or2} can be adapted so that in the statement of Lemma~\ref{lem:or2p} we can also include the edge $uv$, at the cost of two additional handles. Indeed, we can follow the proof of Lemma~\ref{lem:or2} and at the end of the construction we just add the edge $uv$ and two new handles, as illustrated in Figure~\ref{fig:790} (where $uv$ is the  thick black edge). Thus we have the following.

% *********************************************************
\def\tf#1{{\Scale[2.8]{#1}}}
\begin{figure}[ht!]
\centering
\def\hone{{{\mathcal H}_1}}
\def\htwo{{{\mathcal H}_2}}
\def\hthree{{{\mathcal H}_3}}
\def\hfour{{{\mathcal H}_4}}
\def\hfive{{{\mathcal H}_5}}
\def\tc#1{{\Scale[5.5]{#1}}}
\def\tx#1{{\Scale[6]{#1}}}
\def\utf#1{{\Scale[8]{#1}}}
\def\tf#1{{\Scale[4.5]{#1}}}
\def\tg#1{{\Scale[3]{#1}}}
\def\tix#1{{\Scale[5.5]{#1}}}
\def\somea{{\Scale[2.5]{\text{\rm (a)}}}}
\def\someb{{\Scale[2.5]{\text{\rm (b)}}}}
\hglue -0.2cm\scalebox{0.125}{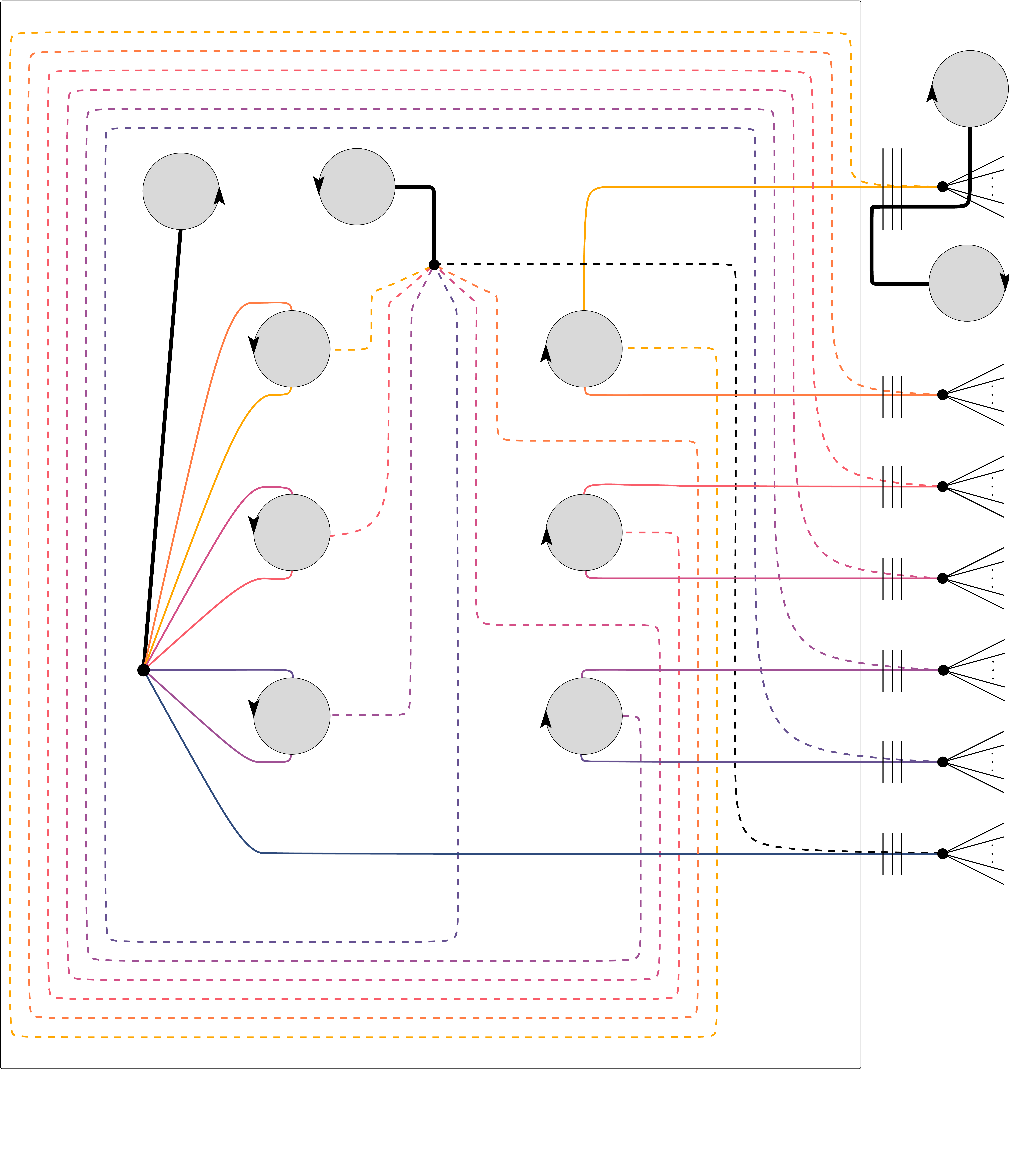}
\caption{Cloning a star (orientable version) and adding an edge.}
\label{fig:790}
\end{figure}
% *********************************************************

\begin{lemma}\label{lem:or2pp}
{Let $G$ be a graph that can be thrackled on some surface $M$. {Let $u$ be a vertex of $G$, and let $d$ be the degree of $u$.} If $G''$ is the graph obtained by adding a vertex $v$ and joining it to $u$ and to all the vertices adjacent to $u$, then $G''$ can be thrackled on {$M\# S_{\ceil{(d-1)/2}+2}$}.}
\end{lemma}

\begin{corollary}\label{cor:Kn}
{For any integer {$n\ge 5$} we have}
\[
\Tg(K_{n})\le 
\begin{cases}
{(n^2 + 4n - 32)/4}, &\mbox{{if $n$ is even}}\\
{(n^2 + 4n - 33)/4}, &\mbox{{if $n$ is odd.}}\\
\end{cases}
\]
\end{corollary}

\begin{proof}
{We proceed by induction on $n$, using {$n=5$ and $n=6$} as base cases.}

{For $n=5$ the corollary claims that $K_5$ can be thrackled in $S_3$, and this is illustrated in Figure~\ref{fig:430}. Now since $K_5$ is $4$-regular and {thrackleable} on $S_3$, Lemma~\ref{lem:or2pp} implies that $\Tg(K_6) \le {3+\ceil{(4-1)/2 + 2}}=7$, as required. Thus the corollary holds for $n=5$ and $n=6$.}

%We note that Cairns, McIntyre, and Nikolayevsky also claimed in~\cite{cairnsk5} that $K_5$ cannot be thrackled on the double torus $S_2$. In other words, they claimed that $\Tg(K_5) > 2$. We do not know whether the incorrectness of their arguments involving $K_{3,3}$ (see the discussion after Conjecture~\ref{con:conj1}) also affect their arguments for $K_5$. In any case, our efforts to produce a thrackle of $K_5$ on $S_2$ have failed, and we strongly believe that indeed no such thrackle exists.}

% *********************************************************
\def\tf#1{{\Scale[2.8]{#1}}}
\begin{figure}[ht!]
\centering
\def\hone{{{\mathcal H}_1}}
\def\htwo{{{\mathcal H}_2}}
\def\hthree{{{\mathcal H}_3}}
\def\hfour{{{\mathcal H}_4}}
\def\hfive{{{\mathcal H}_5}}
\def\tc#1{{\Scale[5.5]{#1}}}
\def\tx#1{{\Scale[6]{#1}}}
\def\utf#1{{\Scale[8]{#1}}}
\def\tf#1{{\Scale[4.5]{#1}}}
\def\tg#1{{\Scale[3]{#1}}}
\def\tix#1{{\Scale[5.5]{#1}}}
\def\somea{{\Scale[2.5]{\text{\rm (a)}}}}
\def\someb{{\Scale[2.5]{\text{\rm (b)}}}}
\hglue -0.2cm\scalebox{0.125}{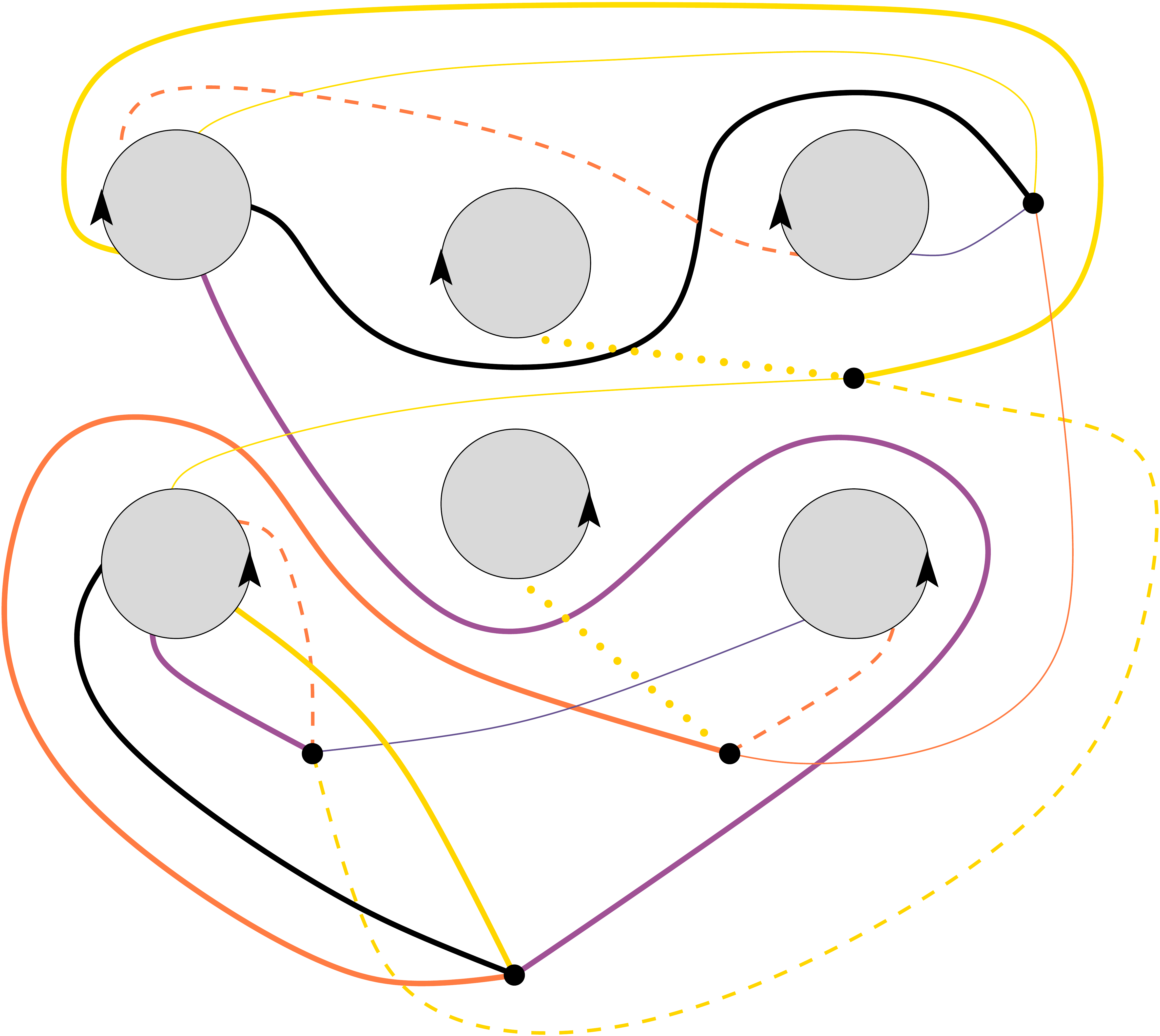}
%\hglue -0.2cm\scalebox{0.125}{\input{430k.pdf_t}}
\caption{A thrackle of $K_5$ on the triple torus $S_3$. Each edge is drawn with one of four colours (black, purple, orange, or gold) and with one of four styles (thick solid, thin solid, dashed, or dotted), in such a way that an edge $ij$ with $i<j$ and an edge $k\ell$ with $k < \ell$ have the same style if and only if $i=j$, and they have the same colour if and only if $k=\ell$. In other words, the smaller endvertex of an edge determines its style, and its larger endvertex determines its colour.}
\label{fig:430}
\end{figure}
% *********************************************************

{For the inductive step we let {$k\ge 6$} be an integer, assume that the statement holds for all {$n=5,6,\ldots,k$}, and {we will show} that then it holds for $n=k+1$.}

{Suppose first that $k+1$ is even. Since $K_k$ is $(k-1)$-regular and (since $k$ is odd) by the induction hypothesis $\Tg(K_{k}) \le {(k^2 + 4k - {33})/4}$, Lemma~\ref{lem:or2pp} implies that $\Tg(K_{k+1})\le (k^2 + 4k - {33})/4 + \ceil{((k-1)-1)/2}  + 2={(k^2 + 4k - {33})/4 + (k-1)/2 + 2}={((k+1)^2 + 4(k+1) - {32})/4}$.}

{Suppose finally that $k+1$ is odd. Since $K_k$ is $(k{-}1)$-regular and (since $k$ is even) by the induction hypothesis $\Tg(K_{k}) \le {(k^2 + 4k - {32})/4}$, Lemma~\ref{lem:or2pp} implies that $\Tg(K_{k+1})\le (k^2 + 4k - {32})/4 + \ceil{((k-1)-1)/2}  + 2 ={(k^2 + 4k - {32})/4 + (k-2)/2 + 2}={((k+1)^2 + 4(k+1) - {33})/4}$.}
\end{proof}

\subsection{{Lower bounds on $\Tg(K_{m,n})$ with $\min\{m,n\}\ge 3$ and on $\Tg(K_n)$}}

{We now show that a result by Cairns and Nikolayevsky and a classical theorem by Ringel easily imply the following lower bound on $\Tg(K_{m,n})$ for $\min\{m,n\} \ge 3$. This shows that $\Tg(K_{m,n})$ is roughly at least half of the upper bound established in Corollary~\ref{cor:kMn}.}

\begin{proposition}\label{pro:lowkmn}
{For any integers $m\ge 3$ and $n\ge 3$ {we have}}
\[
{\Tg(K_{m,n}) \ge \frac{(m-2)(n-2)}{4}.}
\]
\end{proposition}

\begin{proof}
{Cairns and Nikolayevsky proved that a bipartite graph $G$ admits a generalized thrackle on an orientable surface $S$ if and only if $G$ can be embedded on $S$~\cite[Theorem 3]{cairnsbounds}. This implies that if $G$ is bipartite then its thrackle genus is at least its orientable genus $\gamma(G)$. Ringel proved that $\gamma(K_{m,n}) = (m-2)(n-2)/4$ whenever $m\ge 2$ and $n\ge 2$~\cite{RingelBip}, and so it follows that $\Tg(K_{m,n}) \ge (m-2)(n-2)/4$.}
\end{proof}

{We finally note that another result by Cairns and Nikolayevsky (stated above as Theorem~\ref{thm:upgen}) implies the following lower bound on $\Tg(K_{n})$. This shows that $\Tg(K_{n})$ is roughly at least half of the upper bound established in Corollary~\ref{cor:Kn}.}

\begin{proposition}\label{pro:lowkn}
{For all integers $m\ge 3$ and $n\ge 3$ {we have}}
\[
{\Tg(K_{n}) \ge \frac{n^2 -5n + 4}{8}.}
\]
\end{proposition}

\begin{proof}
Theorem~\ref{thm:upgen} implies that if $G$ is a connected graph with $n$ vertices and $m$ edges then $\Tg(G) \ge m/4 - n/2 + 1/2$. Therefore $\Tg(K_n) \ge \binom{n}{2}/4 -n/2 + 1/2 = (n^2 -5n + 4)/8$.
\end{proof}

\subsection{{A lower bound and a conjecture on $\Tg(K_{2,n})$}}
{We start by recalling that Cairns and Nikolayevsky proved that if $T$ is a thrackle of a graph $G$ on some surface, and $C$ is a $4$-cycle of $G$, then the restriction of $T$ to $C$ is nontrivial in $\ZZ_2$-homology~\cite[Lemma 5(a)]{cairnsbounds}.  From this it easily follows that $\Tg(K_{2,n}) = \Omega(\log{n})$. As we now show, a result by Przytycki~{\cite{prz}} implies a better lower bound.
}

\begin{proposition}\label{pro:k2n}
\[
{\Tg(K_{2,n}) = \Omega(n^{1/3}).}
\]
\end{proposition}

\begin{proof}
{Let $T$ be a thrackle of $K_{2,n}$ on $S_g$. If we remove the vertices in the partite class of size $2$, we obtain {a collection $\aa$ of simple arcs that pairwise intersect at exactly two points, hosted at the orientable surface of genus $g$ punctured twice, where the endpoints of each arc in $\aa$ are the two punctures of the surface}. It follows from {a result of Cairns and Nikolayevsky}~\cite[Lemma 5(a)]{cairnsbounds} that the $n$ arcs in $\aa$ are pairwise {non-homologous} mod 2, and so they are pairwise non-homotopic. {By a result of Przytycki~\cite[Theorem 1.5]{prz} then it follows that} $n=O(g^3)$.}
\end{proof}

{Having the nontrivial upper and lower bounds on $\Tg(K_{2,n})$ in Corollary~\ref{cor:k2n} and Proposition~\ref{pro:k2n}, respectively, we slightly lean towards the former one.}

\begin{conjecture}\label{conj:k2n}
{$\Tg(K_{2,n})=\Theta(n)$.}
\end{conjecture}

\subsection{{A final remark and two open questions}}

{As a final observation on the relationship between the thrackle genus and the orientable genus of a graph, let us note that in general $\Tg(G)$ cannot be bounded {from} above by any function of $\gamma(G)$. Indeed, Proposition~\ref{pro:k2n} implies that there are graphs with orientable genus $0$ and arbitrarily large thrackle genus.
}

{We close with two questions on the thrackle genus of arbitrary graphs. {Let us note that the upper bound in Corollary~\ref{cor:Kn} trivially implies the same quadratic upper bound on the thrackle genus of any graph $G$ in the number of its vertices.} This leads us to ask whether the thrackle genus of an arbitrary graph $G$ is at most linear in the number $m$ of its edges.
}

\begin{question}\label{que:open1}
{Is $\Tg(G) = O(m)$?}
\end{question}

\begin{question}\label{que:open2}
How much can the thrackle genus ``jump'' by the addition of an edge? That is, if $G$ is a graph with $n$ vertices and $m$ edges, and $u,v$ are nonadjacent vertices in $G$, what is the largest possible value (in terms of $m$ and $n$) of $\Tg(G + {uv})-\Tg(G)$?
\end{question}

\section*{Acknowledgments}

We thank Kebin Velasquez for helpful discussions. We thank G\'eza T\'oth for providing us with a copy of a paper by Green and Ringeisen. The second author was supported by Project 23-04949X of the Czech Science Foundation (GA\v{C}R). The third author was supported by CONACYT under Proyecto Ciencia de Frontera 191952. 

\bibliographystyle{abbrv}

\bibliography{references.bib}

\end{document}

%% file: 930b.pdf_t
\begin{picture}(0,0)%
\includegraphics{930b.pdf}%
\end{picture}%
\setlength{\unitlength}{4144sp}%
\begingroup\makeatletter\ifx\SetFigFont\undefined%
\gdef\SetFigFont#1#2#3#4#5{%
  \reset@font\fontsize{#1}{#2pt}%
  \fontfamily{#3}\fontseries{#4}\fontshape{#5}%
  \selectfont}%
\fi\endgroup%
\begin{picture}(44983,27030)(-11488,-18611)
\put(26551,6509){\makebox(0,0)[lb]{\smash{{\SetFigFont{20}{24.0}{\familydefault}{\mddefault}{\updefault}{\color[rgb]{0,0,0}$\tf{p}$}%
}}}}
\put(15436,-11761){\makebox(0,0)[lb]{\smash{{\SetFigFont{20}{24.0}{\familydefault}{\mddefault}{\updefault}{\color[rgb]{0,0,0}$\tf{q}$}%
}}}}
\put(5896,-13021){\makebox(0,0)[lb]{\smash{{\SetFigFont{20}{24.0}{\familydefault}{\mddefault}{\updefault}{\color[rgb]{0,0,0}$\tg{Q}$}%
}}}}
\put(31186,-13021){\makebox(0,0)[lb]{\smash{{\SetFigFont{20}{24.0}{\familydefault}{\mddefault}{\updefault}{\color[rgb]{0,0,0}$\tg{Q}$}%
}}}}
\put(22861,-18556){\makebox(0,0)[lb]{\smash{{\SetFigFont{20}{24.0}{\familydefault}{\mddefault}{\updefault}{\color[rgb]{0,0,0}$\someb$}%
}}}}
\put(-2789,-7261){\makebox(0,0)[lb]{\smash{{\SetFigFont{20}{24.0}{\familydefault}{\mddefault}{\updefault}{\color[rgb]{0,0,0}$\tf{t}$}%
}}}}
\put(22546,-7261){\makebox(0,0)[lb]{\smash{{\SetFigFont{20}{24.0}{\familydefault}{\mddefault}{\updefault}{\color[rgb]{0,0,0}$\tf{t}$}%
}}}}
\put(-2429,-18556){\makebox(0,0)[lb]{\smash{{\SetFigFont{20}{24.0}{\familydefault}{\mddefault}{\updefault}{\color[rgb]{0,0,0}$\somea$}%
}}}}
\put(24166,794){\makebox(0,0)[lb]{\smash{{\SetFigFont{20}{24.0}{\familydefault}{\mddefault}{\updefault}{\color[rgb]{0,0,0}$\tf{s}$}%
}}}}
\put(-1034,794){\makebox(0,0)[lb]{\smash{{\SetFigFont{20}{24.0}{\familydefault}{\mddefault}{\updefault}{\color[rgb]{0,0,0}$\tf{s}$}%
}}}}
\end{picture}%

%% file: 130e.pdf_t
\begin{picture}(0,0)%
\includegraphics{130e.pdf}%
\end{picture}%
\setlength{\unitlength}{4144sp}%
\begingroup\makeatletter\ifx\SetFigFont\undefined%
\gdef\SetFigFont#1#2#3#4#5{%
  \reset@font\fontsize{#1}{#2pt}%
  \fontfamily{#3}\fontseries{#4}\fontshape{#5}%
  \selectfont}%
\fi\endgroup%
\begin{picture}(9120,8233)(-239,-7010)
\put(3961,-2761){\makebox(0,0)[lb]{\smash{{\SetFigFont{17}{20.4}{\familydefault}{\mddefault}{\updefault}{\color[rgb]{0,0,0}$\tf{2}$}%
}}}}
\put(8866,-2941){\makebox(0,0)[lb]{\smash{{\SetFigFont{17}{20.4}{\familydefault}{\mddefault}{\updefault}{\color[rgb]{0,0,0}$\tf{1}$}%
}}}}
\put(2116,-6946){\makebox(0,0)[lb]{\smash{{\SetFigFont{17}{20.4}{\familydefault}{\mddefault}{\updefault}{\color[rgb]{0,0,0}$\tf{1}$}%
}}}}
\put(6571,-6946){\makebox(0,0)[lb]{\smash{{\SetFigFont{17}{20.4}{\familydefault}{\mddefault}{\updefault}{\color[rgb]{0,0,0}$\tf{3}$}%
}}}}
\put(-224,-2941){\makebox(0,0)[lb]{\smash{{\SetFigFont{17}{20.4}{\familydefault}{\mddefault}{\updefault}{\color[rgb]{0,0,0}$\tf{3}$}%
}}}}
\put(2026,929){\makebox(0,0)[lb]{\smash{{\SetFigFont{17}{20.4}{\familydefault}{\mddefault}{\updefault}{\color[rgb]{0,0,0}$\tf{1}$}%
}}}}
\put(6661,929){\makebox(0,0)[lb]{\smash{{\SetFigFont{17}{20.4}{\familydefault}{\mddefault}{\updefault}{\color[rgb]{0,0,0}$\tf{3}$}%
}}}}
\put(4591,-4246){\makebox(0,0)[lb]{\smash{{\SetFigFont{17}{20.4}{\familydefault}{\mddefault}{\updefault}{\color[rgb]{0,0,0}$\tf{A}$}%
}}}}
\put(5131,-1996){\makebox(0,0)[lb]{\smash{{\SetFigFont{17}{20.4}{\familydefault}{\mddefault}{\updefault}{\color[rgb]{0,0,0}$\tf{B}$}%
}}}}
\put(3331,-1996){\makebox(0,0)[lb]{\smash{{\SetFigFont{17}{20.4}{\familydefault}{\mddefault}{\updefault}{\color[rgb]{0,0,0}$\tf{C}$}%
}}}}
\put(4771,1064){\makebox(0,0)[lb]{\smash{{\SetFigFont{17}{20.4}{\familydefault}{\mddefault}{\updefault}{\color[rgb]{0,0,0}$\tg{a}$}%
}}}}
\put(4771,-6901){\makebox(0,0)[lb]{\smash{{\SetFigFont{17}{20.4}{\familydefault}{\mddefault}{\updefault}{\color[rgb]{0,0,0}$\tg{a}$}%
}}}}
\put(7966,-1411){\makebox(0,0)[lb]{\smash{{\SetFigFont{17}{20.4}{\familydefault}{\mddefault}{\updefault}{\color[rgb]{0,0,0}$\tg{b}$}%
}}}}
\put(1126,-5326){\makebox(0,0)[lb]{\smash{{\SetFigFont{17}{20.4}{\familydefault}{\mddefault}{\updefault}{\color[rgb]{0,0,0}$\tg{b}$}%
}}}}
\put(676,-1366){\makebox(0,0)[lb]{\smash{{\SetFigFont{17}{20.4}{\familydefault}{\mddefault}{\updefault}{\color[rgb]{0,0,0}$\tg{c}$}%
}}}}
\put(7471,-5326){\makebox(0,0)[lb]{\smash{{\SetFigFont{17}{20.4}{\familydefault}{\mddefault}{\updefault}{\color[rgb]{0,0,0}$\tg{c}$}%
}}}}
\end{picture}%

%% file: 700y.pdf_t
\begin{picture}(0,0)%
\includegraphics{700y.pdf}%
\end{picture}%
\setlength{\unitlength}{4144sp}%
\begingroup\makeatletter\ifx\SetFigFont\undefined%
\gdef\SetFigFont#1#2#3#4#5{%
  \reset@font\fontsize{#1}{#2pt}%
  \fontfamily{#3}\fontseries{#4}\fontshape{#5}%
  \selectfont}%
\fi\endgroup%
\begin{picture}(15378,11906)(-533,-11592)
\put(6931,-11176){\makebox(0,0)[lb]{\smash{{\SetFigFont{20}{24.0}{\familydefault}{\mddefault}{\updefault}{\color[rgb]{0,0,0}$\tf{a}$}%
}}}}
\end{picture}%

%% file: 715c.pdf_t
\begin{picture}(0,0)%
\includegraphics{715c.pdf}%
\end{picture}%
\setlength{\unitlength}{4144sp}%
\begingroup\makeatletter\ifx\SetFigFont\undefined%
\gdef\SetFigFont#1#2#3#4#5{%
  \reset@font\fontsize{#1}{#2pt}%
  \fontfamily{#3}\fontseries{#4}\fontshape{#5}%
  \selectfont}%
\fi\endgroup%
\begin{picture}(60567,82281)(-36437,-48469)
\put(9541,-46321){\makebox(0,0)[lb]{\smash{{\SetFigFont{20}{24.0}{\familydefault}{\mddefault}{\updefault}{\color[rgb]{0,0,0}$\tix{Q}$}%
}}}}
\put(-17549,-38536){\makebox(0,0)[lb]{\smash{{\SetFigFont{20}{24.0}{\familydefault}{\mddefault}{\updefault}{\color[rgb]{0,0,0}$\tif{e_9}$}%
}}}}
\put(-17594,-30076){\makebox(0,0)[lb]{\smash{{\SetFigFont{20}{24.0}{\familydefault}{\mddefault}{\updefault}{\color[rgb]{0,0,0}$\tif{e_8}$}%
}}}}
\put(-17594,-21526){\makebox(0,0)[lb]{\smash{{\SetFigFont{20}{24.0}{\familydefault}{\mddefault}{\updefault}{\color[rgb]{0,0,0}$\tif{e_7}$}%
}}}}
\put(-17639,-13021){\makebox(0,0)[lb]{\smash{{\SetFigFont{20}{24.0}{\familydefault}{\mddefault}{\updefault}{\color[rgb]{0,0,0}$\tif{e_6}$}%
}}}}
\put(-17594,-4561){\makebox(0,0)[lb]{\smash{{\SetFigFont{20}{24.0}{\familydefault}{\mddefault}{\updefault}{\color[rgb]{0,0,0}$\tif{e_5}$}%
}}}}
\put(-17549,4034){\makebox(0,0)[lb]{\smash{{\SetFigFont{20}{24.0}{\familydefault}{\mddefault}{\updefault}{\color[rgb]{0,0,0}$\tif{e_4}$}%
}}}}
\put(-17639,13034){\makebox(0,0)[lb]{\smash{{\SetFigFont{20}{24.0}{\familydefault}{\mddefault}{\updefault}{\color[rgb]{0,0,0}$\tif{e_3}$}%
}}}}
\put(-17549,21314){\makebox(0,0)[lb]{\smash{{\SetFigFont{20}{24.0}{\familydefault}{\mddefault}{\updefault}{\color[rgb]{0,0,0}$\tif{e_2}$}%
}}}}
\put(-17594,29774){\makebox(0,0)[lb]{\smash{{\SetFigFont{20}{24.0}{\familydefault}{\mddefault}{\updefault}{\color[rgb]{0,0,0}$\tif{e_1}$}%
}}}}
\put(19801,21359){\makebox(0,0)[lb]{\smash{{\SetFigFont{20}{24.0}{\familydefault}{\mddefault}{\updefault}{\color[rgb]{0,0,0}$\tf{2}$}%
}}}}
\put(19801,29774){\makebox(0,0)[lb]{\smash{{\SetFigFont{20}{24.0}{\familydefault}{\mddefault}{\updefault}{\color[rgb]{0,0,0}$\tf{1}$}%
}}}}
\put(19801,12944){\makebox(0,0)[lb]{\smash{{\SetFigFont{20}{24.0}{\familydefault}{\mddefault}{\updefault}{\color[rgb]{0,0,0}$\tf{3}$}%
}}}}
\put(19756,4439){\makebox(0,0)[lb]{\smash{{\SetFigFont{20}{24.0}{\familydefault}{\mddefault}{\updefault}{\color[rgb]{0,0,0}$\tf{4}$}%
}}}}
\put(19801,-4111){\makebox(0,0)[lb]{\smash{{\SetFigFont{20}{24.0}{\familydefault}{\mddefault}{\updefault}{\color[rgb]{0,0,0}$\tf{5}$}%
}}}}
\put(19756,-12796){\makebox(0,0)[lb]{\smash{{\SetFigFont{20}{24.0}{\familydefault}{\mddefault}{\updefault}{\color[rgb]{0,0,0}$\tf{6}$}%
}}}}
\put(19801,-21391){\makebox(0,0)[lb]{\smash{{\SetFigFont{20}{24.0}{\familydefault}{\mddefault}{\updefault}{\color[rgb]{0,0,0}$\tf{7}$}%
}}}}
\put(19801,-30031){\makebox(0,0)[lb]{\smash{{\SetFigFont{20}{24.0}{\familydefault}{\mddefault}{\updefault}{\color[rgb]{0,0,0}$\tf{8}$}%
}}}}
\put(19711,-38626){\makebox(0,0)[lb]{\smash{{\SetFigFont{20}{24.0}{\familydefault}{\mddefault}{\updefault}{\color[rgb]{0,0,0}$\tf{9}$}%
}}}}
\put(-30509,-6451){\makebox(0,0)[lb]{\smash{{\SetFigFont{20}{24.0}{\familydefault}{\mddefault}{\updefault}{\color[rgb]{0,0,0}$\tx{u}$}%
}}}}
\end{picture}%

%% file: 745j.pdf_t
\begin{picture}(0,0)%
\includegraphics{745j.pdf}%
\end{picture}%
\setlength{\unitlength}{4144sp}%
\begingroup\makeatletter\ifx\SetFigFont\undefined%
\gdef\SetFigFont#1#2#3#4#5{%
  \reset@font\fontsize{#1}{#2pt}%
  \fontfamily{#3}\fontseries{#4}\fontshape{#5}%
  \selectfont}%
\fi\endgroup%
\begin{picture}(59736,82599)(-35561,-45973)
\put(-28304,-9016){\makebox(0,0)[lb]{\smash{{\SetFigFont{20}{24.0}{\familydefault}{\mddefault}{\updefault}{\color[rgb]{0,0,0}$\tx{u}$}%
}}}}
\put(-17099,-28501){\makebox(0,0)[lb]{\smash{{\SetFigFont{20}{24.0}{\familydefault}{\mddefault}{\updefault}{\color[rgb]{0,0,0}$\tf{e_8}$}%
}}}}
\put(-17099,-21616){\makebox(0,0)[lb]{\smash{{\SetFigFont{20}{24.0}{\familydefault}{\mddefault}{\updefault}{\color[rgb]{0,0,0}$\tf{e_7}$}%
}}}}
\put(20116,-35386){\makebox(0,0)[lb]{\smash{{\SetFigFont{20}{24.0}{\familydefault}{\mddefault}{\updefault}{\color[rgb]{0,0,0}$\tf{9}$}%
}}}}
\put(20071,-28501){\makebox(0,0)[lb]{\smash{{\SetFigFont{20}{24.0}{\familydefault}{\mddefault}{\updefault}{\color[rgb]{0,0,0}$\tf{8}$}%
}}}}
\put(20071,-21616){\makebox(0,0)[lb]{\smash{{\SetFigFont{20}{24.0}{\familydefault}{\mddefault}{\updefault}{\color[rgb]{0,0,0}$\tf{7}$}%
}}}}
\put(20116,-14776){\makebox(0,0)[lb]{\smash{{\SetFigFont{20}{24.0}{\familydefault}{\mddefault}{\updefault}{\color[rgb]{0,0,0}$\tf{6}$}%
}}}}
\put(-9944,29099){\makebox(0,0)[lb]{\smash{{\SetFigFont{20}{24.0}{\familydefault}{\mddefault}{\updefault}{\color[rgb]{0,0,0}$\tx{v}$}%
}}}}
\put(-269,24689){\makebox(0,0)[lb]{\smash{{\SetFigFont{20}{24.0}{\familydefault}{\mddefault}{\updefault}{\color[rgb]{0,0,0}$\tf{f_4}$}%
}}}}
\put(-6839,26669){\makebox(0,0)[lb]{\smash{{\SetFigFont{20}{24.0}{\familydefault}{\mddefault}{\updefault}{\color[rgb]{0,0,0}$\tf{f_2}$}%
}}}}
\put(12826,-44296){\makebox(0,0)[lb]{\smash{{\SetFigFont{20}{24.0}{\familydefault}{\mddefault}{\updefault}{\color[rgb]{0,0,0}$\tix{Q}$}%
}}}}
\put(-17144,-35431){\makebox(0,0)[lb]{\smash{{\SetFigFont{20}{24.0}{\familydefault}{\mddefault}{\updefault}{\color[rgb]{0,0,0}$\tf{e_9}$}%
}}}}
\put(-17144,-14731){\makebox(0,0)[lb]{\smash{{\SetFigFont{20}{24.0}{\familydefault}{\mddefault}{\updefault}{\color[rgb]{0,0,0}$\tf{e_6}$}%
}}}}
\put(-17099,-961){\makebox(0,0)[lb]{\smash{{\SetFigFont{20}{24.0}{\familydefault}{\mddefault}{\updefault}{\color[rgb]{0,0,0}$\tf{e_4}$}%
}}}}
\put(-17144,-7891){\makebox(0,0)[lb]{\smash{{\SetFigFont{20}{24.0}{\familydefault}{\mddefault}{\updefault}{\color[rgb]{0,0,0}$\tf{e_5}$}%
}}}}
\put(-17144,5924){\makebox(0,0)[lb]{\smash{{\SetFigFont{20}{24.0}{\familydefault}{\mddefault}{\updefault}{\color[rgb]{0,0,0}$\tf{e_3}$}%
}}}}
\put(-17099,12764){\makebox(0,0)[lb]{\smash{{\SetFigFont{20}{24.0}{\familydefault}{\mddefault}{\updefault}{\color[rgb]{0,0,0}$\tf{e_2}$}%
}}}}
\put(-10394,-9061){\makebox(0,0)[lb]{\smash{{\SetFigFont{20}{24.0}{\familydefault}{\mddefault}{\updefault}{\color[rgb]{0,0,0}$\tc{\XX_5}$}%
}}}}
\put(-10574,-2266){\makebox(0,0)[lb]{\smash{{\SetFigFont{20}{24.0}{\familydefault}{\mddefault}{\updefault}{\color[rgb]{0,0,0}$\tc{\XX_4}$}%
}}}}
\put(-10394,4664){\makebox(0,0)[lb]{\smash{{\SetFigFont{20}{24.0}{\familydefault}{\mddefault}{\updefault}{\color[rgb]{0,0,0}$\tc{\XX_3}$}%
}}}}
\put(-10439,11549){\makebox(0,0)[lb]{\smash{{\SetFigFont{20}{24.0}{\familydefault}{\mddefault}{\updefault}{\color[rgb]{0,0,0}$\tc{\XX_2}$}%
}}}}
\put(-10394,18389){\makebox(0,0)[lb]{\smash{{\SetFigFont{20}{24.0}{\familydefault}{\mddefault}{\updefault}{\color[rgb]{0,0,0}$\tc{\XX_1}$}%
}}}}
\put(-20384,24689){\makebox(0,0)[lb]{\smash{{\SetFigFont{20}{24.0}{\familydefault}{\mddefault}{\updefault}{\color[rgb]{0,0,0}$\tf{f_3}$}%
}}}}
\put(-13454,26759){\makebox(0,0)[lb]{\smash{{\SetFigFont{20}{24.0}{\familydefault}{\mddefault}{\updefault}{\color[rgb]{0,0,0}$\tf{f_1}$}%
}}}}
\put(4276,26984){\makebox(0,0)[lb]{\smash{{\SetFigFont{20}{24.0}{\familydefault}{\mddefault}{\updefault}{\color[rgb]{0,0,0}$\tf{f_7}$}%
}}}}
\put(-17639,23204){\makebox(0,0)[lb]{\smash{{\SetFigFont{20}{24.0}{\familydefault}{\mddefault}{\updefault}{\color[rgb]{0,0,0}$\tf{f_5}$}%
}}}}
\put(-17144,19649){\makebox(0,0)[lb]{\smash{{\SetFigFont{20}{24.0}{\familydefault}{\mddefault}{\updefault}{\color[rgb]{0,0,0}$\tf{e_1}$}%
}}}}
\put(-9314,-39841){\makebox(0,0)[lb]{\smash{{\SetFigFont{20}{24.0}{\familydefault}{\mddefault}{\updefault}{\color[rgb]{0,0,0}$\tf{f_5}$}%
}}}}
\put(  1,14924){\makebox(0,0)[lb]{\smash{{\SetFigFont{20}{24.0}{\familydefault}{\mddefault}{\updefault}{\color[rgb]{0,0,0}$\tf{f_1}$}%
}}}}
\put(  1,8039){\makebox(0,0)[lb]{\smash{{\SetFigFont{20}{24.0}{\familydefault}{\mddefault}{\updefault}{\color[rgb]{0,0,0}$\tf{f_2}$}%
}}}}
\put( 46,1064){\makebox(0,0)[lb]{\smash{{\SetFigFont{20}{24.0}{\familydefault}{\mddefault}{\updefault}{\color[rgb]{0,0,0}$\tf{f_3}$}%
}}}}
\put(  1,-5776){\makebox(0,0)[lb]{\smash{{\SetFigFont{20}{24.0}{\familydefault}{\mddefault}{\updefault}{\color[rgb]{0,0,0}$\tf{f_4}$}%
}}}}
\put(20071,-7891){\makebox(0,0)[lb]{\smash{{\SetFigFont{20}{24.0}{\familydefault}{\mddefault}{\updefault}{\color[rgb]{0,0,0}$\tf{5}$}%
}}}}
\put(20116,-1006){\makebox(0,0)[lb]{\smash{{\SetFigFont{20}{24.0}{\familydefault}{\mddefault}{\updefault}{\color[rgb]{0,0,0}$\tf{4}$}%
}}}}
\put(20071,5879){\makebox(0,0)[lb]{\smash{{\SetFigFont{20}{24.0}{\familydefault}{\mddefault}{\updefault}{\color[rgb]{0,0,0}$\tf{3}$}%
}}}}
\put(20071,12764){\makebox(0,0)[lb]{\smash{{\SetFigFont{20}{24.0}{\familydefault}{\mddefault}{\updefault}{\color[rgb]{0,0,0}$\tf{2}$}%
}}}}
\put(20071,19649){\makebox(0,0)[lb]{\smash{{\SetFigFont{20}{24.0}{\familydefault}{\mddefault}{\updefault}{\color[rgb]{0,0,0}$\tf{1}$}%
}}}}
\end{picture}%

%% file: 665e.pdf_t
\begin{picture}(0,0)%
\includegraphics{665e.pdf}%
\end{picture}%
\setlength{\unitlength}{4144sp}%
\begingroup\makeatletter\ifx\SetFigFont\undefined%
\gdef\SetFigFont#1#2#3#4#5{%
  \reset@font\fontsize{#1}{#2pt}%
  \fontfamily{#3}\fontseries{#4}\fontshape{#5}%
  \selectfont}%
\fi\endgroup%
\begin{picture}(31946,8581)(166,-8000)
\put(18091,-2941){\makebox(0,0)[lb]{\smash{{\SetFigFont{20}{24.0}{\familydefault}{\mddefault}{\updefault}{\color[rgb]{0,0,0}$\tf{6}$}%
}}}}
\put(25381,-2941){\makebox(0,0)[lb]{\smash{{\SetFigFont{20}{24.0}{\familydefault}{\mddefault}{\updefault}{\color[rgb]{0,0,0}$\tf{8}$}%
}}}}
\put(28891,-4561){\makebox(0,0)[lb]{\smash{{\SetFigFont{20}{24.0}{\familydefault}{\mddefault}{\updefault}{\color[rgb]{0,0,0}$\tf{9}$}%
}}}}
\put(31456,-2761){\makebox(0,0)[lb]{\smash{{\SetFigFont{20}{24.0}{\familydefault}{\mddefault}{\updefault}{\color[rgb]{0,0,0}$\tf{10}$}%
}}}}
\put(21826,-4561){\makebox(0,0)[lb]{\smash{{\SetFigFont{20}{24.0}{\familydefault}{\mddefault}{\updefault}{\color[rgb]{0,0,0}$\tf{7}$}%
}}}}
\put(14626,-4561){\makebox(0,0)[lb]{\smash{{\SetFigFont{20}{24.0}{\familydefault}{\mddefault}{\updefault}{\color[rgb]{0,0,0}$\tf{5}$}%
}}}}
\put(10846,-2851){\makebox(0,0)[lb]{\smash{{\SetFigFont{20}{24.0}{\familydefault}{\mddefault}{\updefault}{\color[rgb]{0,0,0}$\tf{4}$}%
}}}}
\put(7336,-4516){\makebox(0,0)[lb]{\smash{{\SetFigFont{20}{24.0}{\familydefault}{\mddefault}{\updefault}{\color[rgb]{0,0,0}$\tf{3}$}%
}}}}
\put(3826,-2941){\makebox(0,0)[lb]{\smash{{\SetFigFont{20}{24.0}{\familydefault}{\mddefault}{\updefault}{\color[rgb]{0,0,0}$\tf{2}$}%
}}}}
\put(181,-4516){\makebox(0,0)[lb]{\smash{{\SetFigFont{20}{24.0}{\familydefault}{\mddefault}{\updefault}{\color[rgb]{0,0,0}$\tf{1}$}%
}}}}
\put(12421,-7936){\makebox(0,0)[lb]{\smash{{\SetFigFont{20}{24.0}{\familydefault}{\mddefault}{\updefault}{\color[rgb]{0,0,0}$\tf{b}$}%
}}}}
\put(12646,434){\makebox(0,0)[lb]{\smash{{\SetFigFont{20}{24.0}{\familydefault}{\mddefault}{\updefault}{\color[rgb]{0,0,0}$\tf{a}$}%
}}}}
\end{picture}%

%% file: 670i.pdf_t
\begin{picture}(0,0)%
\includegraphics{670i.pdf}%
\end{picture}%
\setlength{\unitlength}{4144sp}%
\begingroup\makeatletter\ifx\SetFigFont\undefined%
\gdef\SetFigFont#1#2#3#4#5{%
  \reset@font\fontsize{#1}{#2pt}%
  \fontfamily{#3}\fontseries{#4}\fontshape{#5}%
  \selectfont}%
\fi\endgroup%
\begin{picture}(32205,29071)(24873,-26775)
\put(29836,-12931){\makebox(0,0)[lb]{\smash{{\SetFigFont{20}{24.0}{\familydefault}{\mddefault}{\updefault}{\color[rgb]{0,0,0}$\tf{7}$}%
}}}}
\put(27136,-18871){\makebox(0,0)[lb]{\smash{{\SetFigFont{20}{24.0}{\familydefault}{\mddefault}{\updefault}{\color[rgb]{0,0,0}$\tf{1}$}%
}}}}
\put(27991,-16621){\makebox(0,0)[lb]{\smash{{\SetFigFont{20}{24.0}{\familydefault}{\mddefault}{\updefault}{\color[rgb]{0,0,0}$\tf{3}$}%
}}}}
\put(28936,-14731){\makebox(0,0)[lb]{\smash{{\SetFigFont{20}{24.0}{\familydefault}{\mddefault}{\updefault}{\color[rgb]{0,0,0}$\tf{5}$}%
}}}}
\put(30691,-11176){\makebox(0,0)[lb]{\smash{{\SetFigFont{20}{24.0}{\familydefault}{\mddefault}{\updefault}{\color[rgb]{0,0,0}$\tf{9}$}%
}}}}
\put(43201,-10276){\makebox(0,0)[lb]{\smash{{\SetFigFont{20}{24.0}{\familydefault}{\mddefault}{\updefault}{\color[rgb]{0,0,0}$\tf{10}$}%
}}}}
\put(42346,-12076){\makebox(0,0)[lb]{\smash{{\SetFigFont{20}{24.0}{\familydefault}{\mddefault}{\updefault}{\color[rgb]{0,0,0}$\tf{8}$}%
}}}}
\put(41401,-13831){\makebox(0,0)[lb]{\smash{{\SetFigFont{20}{24.0}{\familydefault}{\mddefault}{\updefault}{\color[rgb]{0,0,0}$\tf{6}$}%
}}}}
\put(40546,-15586){\makebox(0,0)[lb]{\smash{{\SetFigFont{20}{24.0}{\familydefault}{\mddefault}{\updefault}{\color[rgb]{0,0,0}$\tf{4}$}%
}}}}
\put(39646,-17431){\makebox(0,0)[lb]{\smash{{\SetFigFont{20}{24.0}{\familydefault}{\mddefault}{\updefault}{\color[rgb]{0,0,0}$\tf{2}$}%
}}}}
\put(47116,-22066){\makebox(0,0)[lb]{\smash{{\SetFigFont{20}{24.0}{\familydefault}{\mddefault}{\updefault}{\color[rgb]{0,0,0}$\tf{a}$}%
}}}}
\put(36451,-21346){\makebox(0,0)[lb]{\smash{{\SetFigFont{20}{24.0}{\familydefault}{\mddefault}{\updefault}{\color[rgb]{0,0,0}$\tf{b}$}%
}}}}
\end{picture}%

%% file: 740i.pdf_t
\begin{picture}(0,0)%
\includegraphics{740i.pdf}%
\end{picture}%
\setlength{\unitlength}{4144sp}%
\begingroup\makeatletter\ifx\SetFigFont\undefined%
\gdef\SetFigFont#1#2#3#4#5{%
  \reset@font\fontsize{#1}{#2pt}%
  \fontfamily{#3}\fontseries{#4}\fontshape{#5}%
  \selectfont}%
\fi\endgroup%
\begin{picture}(59667,64461)(-35492,-37264)
\put(20071,16139){\makebox(0,0)[lb]{\smash{{\SetFigFont{20}{24.0}{\familydefault}{\mddefault}{\updefault}{\color[rgb]{0,0,0}$\tf{1}$}%
}}}}
\put(20071,-27061){\makebox(0,0)[lb]{\smash{{\SetFigFont{20}{24.0}{\familydefault}{\mddefault}{\updefault}{\color[rgb]{0,0,0}$\tf{9}$}%
}}}}
\put(20071,-21661){\makebox(0,0)[lb]{\smash{{\SetFigFont{20}{24.0}{\familydefault}{\mddefault}{\updefault}{\color[rgb]{0,0,0}$\tf{8}$}%
}}}}
\put(20071,-16261){\makebox(0,0)[lb]{\smash{{\SetFigFont{20}{24.0}{\familydefault}{\mddefault}{\updefault}{\color[rgb]{0,0,0}$\tf{7}$}%
}}}}
\put(20071,-10861){\makebox(0,0)[lb]{\smash{{\SetFigFont{20}{24.0}{\familydefault}{\mddefault}{\updefault}{\color[rgb]{0,0,0}$\tf{6}$}%
}}}}
\put(20071,-61){\makebox(0,0)[lb]{\smash{{\SetFigFont{20}{24.0}{\familydefault}{\mddefault}{\updefault}{\color[rgb]{0,0,0}$\tf{4}$}%
}}}}
\put(20116,-5461){\makebox(0,0)[lb]{\smash{{\SetFigFont{20}{24.0}{\familydefault}{\mddefault}{\updefault}{\color[rgb]{0,0,0}$\tf{5}$}%
}}}}
\put(20071,10739){\makebox(0,0)[lb]{\smash{{\SetFigFont{20}{24.0}{\familydefault}{\mddefault}{\updefault}{\color[rgb]{0,0,0}$\tf{2}$}%
}}}}
\put(20071,5339){\makebox(0,0)[lb]{\smash{{\SetFigFont{20}{24.0}{\familydefault}{\mddefault}{\updefault}{\color[rgb]{0,0,0}$\tf{3}$}%
}}}}
\put(-14129,11504){\makebox(0,0)[lb]{\smash{{\SetFigFont{20}{24.0}{\familydefault}{\mddefault}{\updefault}{\color[rgb]{0,0,0}$\tf{f_1}$}%
}}}}
\put(-5354,6149){\makebox(0,0)[lb]{\smash{{\SetFigFont{20}{24.0}{\familydefault}{\mddefault}{\updefault}{\color[rgb]{0,0,0}$\tf{f_2}$}%
}}}}
\put(-14129,884){\makebox(0,0)[lb]{\smash{{\SetFigFont{20}{24.0}{\familydefault}{\mddefault}{\updefault}{\color[rgb]{0,0,0}$\tf{f_3}$}%
}}}}
\put(-5354,-4786){\makebox(0,0)[lb]{\smash{{\SetFigFont{20}{24.0}{\familydefault}{\mddefault}{\updefault}{\color[rgb]{0,0,0}$\tf{f_4}$}%
}}}}
\put(-14129,-10186){\makebox(0,0)[lb]{\smash{{\SetFigFont{20}{24.0}{\familydefault}{\mddefault}{\updefault}{\color[rgb]{0,0,0}$\tf{f_5}$}%
}}}}
\put(-7829,-12661){\makebox(0,0)[lb]{\smash{{\SetFigFont{20}{24.0}{\familydefault}{\mddefault}{\updefault}{\color[rgb]{0,0,0}$\tf{f_6}$}%
}}}}
\put(-1619,-9196){\makebox(0,0)[lb]{\smash{{\SetFigFont{20}{24.0}{\familydefault}{\mddefault}{\updefault}{\color[rgb]{0,0,0}$\tc{\hthree}$}%
}}}}
\put(-18809,-9196){\makebox(0,0)[lb]{\smash{{\SetFigFont{20}{24.0}{\familydefault}{\mddefault}{\updefault}{\color[rgb]{0,0,0}$\tc{\hthree}$}%
}}}}
\put(-18809,1604){\makebox(0,0)[lb]{\smash{{\SetFigFont{20}{24.0}{\familydefault}{\mddefault}{\updefault}{\color[rgb]{0,0,0}$\tc{\htwo}$}%
}}}}
\put(-1619,1604){\makebox(0,0)[lb]{\smash{{\SetFigFont{20}{24.0}{\familydefault}{\mddefault}{\updefault}{\color[rgb]{0,0,0}$\tc{\htwo}$}%
}}}}
\put(-20519,7769){\makebox(0,0)[lb]{\smash{{\SetFigFont{20}{24.0}{\familydefault}{\mddefault}{\updefault}{\color[rgb]{0,0,0}$\tf{e_1}$}%
}}}}
\put(-15614,-27061){\makebox(0,0)[lb]{\smash{{\SetFigFont{20}{24.0}{\familydefault}{\mddefault}{\updefault}{\color[rgb]{0,0,0}$\tf{e_9}$}%
}}}}
\put(-15434,-21751){\makebox(0,0)[lb]{\smash{{\SetFigFont{20}{24.0}{\familydefault}{\mddefault}{\updefault}{\color[rgb]{0,0,0}$\tf{e_8}$}%
}}}}
\put(-15344,-16306){\makebox(0,0)[lb]{\smash{{\SetFigFont{20}{24.0}{\familydefault}{\mddefault}{\updefault}{\color[rgb]{0,0,0}$\tf{e_7}$}%
}}}}
\put(-22859,-556){\makebox(0,0)[lb]{\smash{{\SetFigFont{20}{24.0}{\familydefault}{\mddefault}{\updefault}{\color[rgb]{0,0,0}$\tf{e_4}$}%
}}}}
\put(-22904,-3796){\makebox(0,0)[lb]{\smash{{\SetFigFont{20}{24.0}{\familydefault}{\mddefault}{\updefault}{\color[rgb]{0,0,0}$\tf{e_3}$}%
}}}}
\put(-22634,-11356){\makebox(0,0)[lb]{\smash{{\SetFigFont{20}{24.0}{\familydefault}{\mddefault}{\updefault}{\color[rgb]{0,0,0}$\tf{e_5}$}%
}}}}
\put(-22994,-7126){\makebox(0,0)[lb]{\smash{{\SetFigFont{20}{24.0}{\familydefault}{\mddefault}{\updefault}{\color[rgb]{0,0,0}$\tf{e_6}$}%
}}}}
\put(-25424,7724){\makebox(0,0)[lb]{\smash{{\SetFigFont{20}{24.0}{\familydefault}{\mddefault}{\updefault}{\color[rgb]{0,0,0}$\tf{e_2}$}%
}}}}
\put(-18809,12404){\makebox(0,0)[lb]{\smash{{\SetFigFont{20}{24.0}{\familydefault}{\mddefault}{\updefault}{\color[rgb]{0,0,0}$\tc{\hone}$}%
}}}}
\put(-1619,12404){\makebox(0,0)[lb]{\smash{{\SetFigFont{20}{24.0}{\familydefault}{\mddefault}{\updefault}{\color[rgb]{0,0,0}$\tc{\hone}$}%
}}}}
\put(-9944,18524){\makebox(0,0)[lb]{\smash{{\SetFigFont{20}{24.0}{\familydefault}{\mddefault}{\updefault}{\color[rgb]{0,0,0}$\tx{v}$}%
}}}}
\put(-4004,16544){\makebox(0,0)[lb]{\smash{{\SetFigFont{20}{24.0}{\familydefault}{\mddefault}{\updefault}{\color[rgb]{0,0,0}$\tf{f_9}$}%
}}}}
\put(-28304,-6541){\makebox(0,0)[lb]{\smash{{\SetFigFont{20}{24.0}{\familydefault}{\mddefault}{\updefault}{\color[rgb]{0,0,0}$\tx{u}$}%
}}}}
\put(3151,13394){\makebox(0,0)[lb]{\smash{{\SetFigFont{20}{24.0}{\familydefault}{\mddefault}{\updefault}{\color[rgb]{0,0,0}$\tf{f_1}$}%
}}}}
\put(3151,2459){\makebox(0,0)[lb]{\smash{{\SetFigFont{20}{24.0}{\familydefault}{\mddefault}{\updefault}{\color[rgb]{0,0,0}$\tf{f_3}$}%
}}}}
\put(1936,-8386){\makebox(0,0)[lb]{\smash{{\SetFigFont{20}{24.0}{\familydefault}{\mddefault}{\updefault}{\color[rgb]{0,0,0}$\tf{f_5}$}%
}}}}
\put(12511,-35701){\makebox(0,0)[lb]{\smash{{\SetFigFont{20}{24.0}{\familydefault}{\mddefault}{\updefault}{\color[rgb]{0,0,0}$\tix{Q}$}%
}}}}
\end{picture}%

%% file: 165a.pdf_t
\begin{picture}(0,0)%
\includegraphics{165a.pdf}%
\end{picture}%
\setlength{\unitlength}{4144sp}%
\begingroup\makeatletter\ifx\SetFigFont\undefined%
\gdef\SetFigFont#1#2#3#4#5{%
  \reset@font\fontsize{#1}{#2pt}%
  \fontfamily{#3}\fontseries{#4}\fontshape{#5}%
  \selectfont}%
\fi\endgroup%
\begin{picture}(9075,8193)(-194,-6970)
\put(4771,1064){\makebox(0,0)[lb]{\smash{{\SetFigFont{17}{20.4}{\familydefault}{\mddefault}{\updefault}{\color[rgb]{0,0,0}$\tg{a}$}%
}}}}
\put(4771,-6901){\makebox(0,0)[lb]{\smash{{\SetFigFont{17}{20.4}{\familydefault}{\mddefault}{\updefault}{\color[rgb]{0,0,0}$\tg{a}$}%
}}}}
\put(7966,-1411){\makebox(0,0)[lb]{\smash{{\SetFigFont{17}{20.4}{\familydefault}{\mddefault}{\updefault}{\color[rgb]{0,0,0}$\tg{b}$}%
}}}}
\put(1126,-5326){\makebox(0,0)[lb]{\smash{{\SetFigFont{17}{20.4}{\familydefault}{\mddefault}{\updefault}{\color[rgb]{0,0,0}$\tg{b}$}%
}}}}
\put(676,-1366){\makebox(0,0)[lb]{\smash{{\SetFigFont{17}{20.4}{\familydefault}{\mddefault}{\updefault}{\color[rgb]{0,0,0}$\tg{c}$}%
}}}}
\put(7471,-5326){\makebox(0,0)[lb]{\smash{{\SetFigFont{17}{20.4}{\familydefault}{\mddefault}{\updefault}{\color[rgb]{0,0,0}$\tg{c}$}%
}}}}
\put(3961,-2761){\makebox(0,0)[lb]{\smash{{\SetFigFont{17}{20.4}{\familydefault}{\mddefault}{\updefault}{\color[rgb]{0,0,0}$\tf{1}$}%
}}}}
\put(4816,-4786){\makebox(0,0)[lb]{\smash{{\SetFigFont{17}{20.4}{\familydefault}{\mddefault}{\updefault}{\color[rgb]{0,0,0}$\tf{4}$}%
}}}}
\put(1891,929){\makebox(0,0)[lb]{\smash{{\SetFigFont{17}{20.4}{\familydefault}{\mddefault}{\updefault}{\color[rgb]{0,0,0}$\tf{x}$}%
}}}}
\put(-179,-2851){\makebox(0,0)[lb]{\smash{{\SetFigFont{17}{20.4}{\familydefault}{\mddefault}{\updefault}{\color[rgb]{0,0,0}$\tf{y}$}%
}}}}
\put(8866,-2851){\makebox(0,0)[lb]{\smash{{\SetFigFont{17}{20.4}{\familydefault}{\mddefault}{\updefault}{\color[rgb]{0,0,0}$\tf{x}$}%
}}}}
\put(2116,-6856){\makebox(0,0)[lb]{\smash{{\SetFigFont{17}{20.4}{\familydefault}{\mddefault}{\updefault}{\color[rgb]{0,0,0}$\tf{x}$}%
}}}}
\put(6571,-6811){\makebox(0,0)[lb]{\smash{{\SetFigFont{17}{20.4}{\familydefault}{\mddefault}{\updefault}{\color[rgb]{0,0,0}$\tf{y}$}%
}}}}
\put(6796,929){\makebox(0,0)[lb]{\smash{{\SetFigFont{17}{20.4}{\familydefault}{\mddefault}{\updefault}{\color[rgb]{0,0,0}$\tf{y}$}%
}}}}
\put(2431,-2401){\makebox(0,0)[lb]{\smash{{\SetFigFont{17}{20.4}{\familydefault}{\mddefault}{\updefault}{\color[rgb]{0,0,0}$\tf{2}$}%
}}}}
\put(6166,-2401){\makebox(0,0)[lb]{\smash{{\SetFigFont{17}{20.4}{\familydefault}{\mddefault}{\updefault}{\color[rgb]{0,0,0}$\tf{3}$}%
}}}}
\end{picture}%

%% file: 790a.pdf_t
\begin{picture}(0,0)%
\includegraphics{790a.pdf}%
\end{picture}%
\setlength{\unitlength}{4144sp}%
\begingroup\makeatletter\ifx\SetFigFont\undefined%
\gdef\SetFigFont#1#2#3#4#5{%
  \reset@font\fontsize{#1}{#2pt}%
  \fontfamily{#3}\fontseries{#4}\fontshape{#5}%
  \selectfont}%
\fi\endgroup%
\begin{picture}(59366,69154)(-34997,-35837)
\put(-1619,12404){\makebox(0,0)[lb]{\smash{{\SetFigFont{20}{24.0}{\familydefault}{\mddefault}{\updefault}{\color[rgb]{0,0,0}$\tc{\hone}$}%
}}}}
\put(20071,22979){\makebox(0,0)[lb]{\smash{{\SetFigFont{20}{24.0}{\familydefault}{\mddefault}{\updefault}{\color[rgb]{0,0,0}$\tf{1}$}%
}}}}
\put(21106,27704){\makebox(0,0)[lb]{\smash{{\SetFigFont{20}{24.0}{\familydefault}{\mddefault}{\updefault}{\color[rgb]{0,0,0}$\tc{\hfour}$}%
}}}}
\put(20926,16274){\makebox(0,0)[lb]{\smash{{\SetFigFont{20}{24.0}{\familydefault}{\mddefault}{\updefault}{\color[rgb]{0,0,0}$\tc{\hfive}$}%
}}}}
\put(-23354,21674){\makebox(0,0)[rb]{\smash{{\SetFigFont{20}{24.0}{\familydefault}{\mddefault}{\updefault}{\color[rgb]{0,0,0}$\tc{\hfour}$}%
}}}}
\put(-13004,21944){\makebox(0,0)[rb]{\smash{{\SetFigFont{20}{24.0}{\familydefault}{\mddefault}{\updefault}{\color[rgb]{0,0,0}$\tc{\hfive}$}%
}}}}
\put(20071,-16261){\makebox(0,0)[lb]{\smash{{\SetFigFont{20}{24.0}{\familydefault}{\mddefault}{\updefault}{\color[rgb]{0,0,0}$\tf{7}$}%
}}}}
\put(20071,-10861){\makebox(0,0)[lb]{\smash{{\SetFigFont{20}{24.0}{\familydefault}{\mddefault}{\updefault}{\color[rgb]{0,0,0}$\tf{6}$}%
}}}}
\put(20071,-61){\makebox(0,0)[lb]{\smash{{\SetFigFont{20}{24.0}{\familydefault}{\mddefault}{\updefault}{\color[rgb]{0,0,0}$\tf{4}$}%
}}}}
\put(20116,-5461){\makebox(0,0)[lb]{\smash{{\SetFigFont{20}{24.0}{\familydefault}{\mddefault}{\updefault}{\color[rgb]{0,0,0}$\tf{5}$}%
}}}}
\put(20071,10739){\makebox(0,0)[lb]{\smash{{\SetFigFont{20}{24.0}{\familydefault}{\mddefault}{\updefault}{\color[rgb]{0,0,0}$\tf{2}$}%
}}}}
\put(20071,5339){\makebox(0,0)[lb]{\smash{{\SetFigFont{20}{24.0}{\familydefault}{\mddefault}{\updefault}{\color[rgb]{0,0,0}$\tf{3}$}%
}}}}
\put(-14129,11504){\makebox(0,0)[lb]{\smash{{\SetFigFont{20}{24.0}{\familydefault}{\mddefault}{\updefault}{\color[rgb]{0,0,0}$\tf{f_1}$}%
}}}}
\put(-5354,6149){\makebox(0,0)[lb]{\smash{{\SetFigFont{20}{24.0}{\familydefault}{\mddefault}{\updefault}{\color[rgb]{0,0,0}$\tf{f_2}$}%
}}}}
\put(-14129,884){\makebox(0,0)[lb]{\smash{{\SetFigFont{20}{24.0}{\familydefault}{\mddefault}{\updefault}{\color[rgb]{0,0,0}$\tf{f_3}$}%
}}}}
\put(-5354,-4786){\makebox(0,0)[lb]{\smash{{\SetFigFont{20}{24.0}{\familydefault}{\mddefault}{\updefault}{\color[rgb]{0,0,0}$\tf{f_4}$}%
}}}}
\put(-14129,-10186){\makebox(0,0)[lb]{\smash{{\SetFigFont{20}{24.0}{\familydefault}{\mddefault}{\updefault}{\color[rgb]{0,0,0}$\tf{f_5}$}%
}}}}
\put(-7829,-12661){\makebox(0,0)[lb]{\smash{{\SetFigFont{20}{24.0}{\familydefault}{\mddefault}{\updefault}{\color[rgb]{0,0,0}$\tf{f_6}$}%
}}}}
\put(-1619,-9196){\makebox(0,0)[lb]{\smash{{\SetFigFont{20}{24.0}{\familydefault}{\mddefault}{\updefault}{\color[rgb]{0,0,0}$\tc{\hthree}$}%
}}}}
\put(-18809,-9196){\makebox(0,0)[lb]{\smash{{\SetFigFont{20}{24.0}{\familydefault}{\mddefault}{\updefault}{\color[rgb]{0,0,0}$\tc{\hthree}$}%
}}}}
\put(-18809,1604){\makebox(0,0)[lb]{\smash{{\SetFigFont{20}{24.0}{\familydefault}{\mddefault}{\updefault}{\color[rgb]{0,0,0}$\tc{\htwo}$}%
}}}}
\put(-1619,1604){\makebox(0,0)[lb]{\smash{{\SetFigFont{20}{24.0}{\familydefault}{\mddefault}{\updefault}{\color[rgb]{0,0,0}$\tc{\htwo}$}%
}}}}
\put(-20519,7769){\makebox(0,0)[lb]{\smash{{\SetFigFont{20}{24.0}{\familydefault}{\mddefault}{\updefault}{\color[rgb]{0,0,0}$\tf{e_1}$}%
}}}}
\put(-15344,-16306){\makebox(0,0)[lb]{\smash{{\SetFigFont{20}{24.0}{\familydefault}{\mddefault}{\updefault}{\color[rgb]{0,0,0}$\tf{e_7}$}%
}}}}
\put(-22859,-556){\makebox(0,0)[lb]{\smash{{\SetFigFont{20}{24.0}{\familydefault}{\mddefault}{\updefault}{\color[rgb]{0,0,0}$\tf{e_4}$}%
}}}}
\put(-22904,-3796){\makebox(0,0)[lb]{\smash{{\SetFigFont{20}{24.0}{\familydefault}{\mddefault}{\updefault}{\color[rgb]{0,0,0}$\tf{e_3}$}%
}}}}
\put(-22634,-11356){\makebox(0,0)[lb]{\smash{{\SetFigFont{20}{24.0}{\familydefault}{\mddefault}{\updefault}{\color[rgb]{0,0,0}$\tf{e_5}$}%
}}}}
\put(-22994,-7126){\makebox(0,0)[lb]{\smash{{\SetFigFont{20}{24.0}{\familydefault}{\mddefault}{\updefault}{\color[rgb]{0,0,0}$\tf{e_6}$}%
}}}}
\put(-18809,12404){\makebox(0,0)[lb]{\smash{{\SetFigFont{20}{24.0}{\familydefault}{\mddefault}{\updefault}{\color[rgb]{0,0,0}$\tc{\hone}$}%
}}}}
\put(-4004,16544){\makebox(0,0)[lb]{\smash{{\SetFigFont{20}{24.0}{\familydefault}{\mddefault}{\updefault}{\color[rgb]{0,0,0}$\tf{f_9}$}%
}}}}
\put(-28304,-6541){\makebox(0,0)[lb]{\smash{{\SetFigFont{20}{24.0}{\familydefault}{\mddefault}{\updefault}{\color[rgb]{0,0,0}$\tx{u}$}%
}}}}
\put(3151,13394){\makebox(0,0)[lb]{\smash{{\SetFigFont{20}{24.0}{\familydefault}{\mddefault}{\updefault}{\color[rgb]{0,0,0}$\tf{f_1}$}%
}}}}
\put(3151,2459){\makebox(0,0)[lb]{\smash{{\SetFigFont{20}{24.0}{\familydefault}{\mddefault}{\updefault}{\color[rgb]{0,0,0}$\tf{f_3}$}%
}}}}
\put(1936,-8386){\makebox(0,0)[lb]{\smash{{\SetFigFont{20}{24.0}{\familydefault}{\mddefault}{\updefault}{\color[rgb]{0,0,0}$\tf{f_5}$}%
}}}}
\put(12511,-35701){\makebox(0,0)[lb]{\smash{{\SetFigFont{20}{24.0}{\familydefault}{\mddefault}{\updefault}{\color[rgb]{0,0,0}$\tix{Q}$}%
}}}}
\put(-24974,7724){\makebox(0,0)[lb]{\smash{{\SetFigFont{20}{24.0}{\familydefault}{\mddefault}{\updefault}{\color[rgb]{0,0,0}$\tf{e_2}$}%
}}}}
\put(-8774,18659){\makebox(0,0)[lb]{\smash{{\SetFigFont{20}{24.0}{\familydefault}{\mddefault}{\updefault}{\color[rgb]{0,0,0}$\tx{v}$}%
}}}}
\end{picture}%

%% file: 430j.pdf_t
\begin{picture}(0,0)%
\includegraphics{430j.pdf}%
\end{picture}%
\setlength{\unitlength}{4144sp}%
\begingroup\makeatletter\ifx\SetFigFont\undefined%
\gdef\SetFigFont#1#2#3#4#5{%
  \reset@font\fontsize{#1}{#2pt}%
  \fontfamily{#3}\fontseries{#4}\fontshape{#5}%
  \selectfont}%
\fi\endgroup%
\begin{picture}(34893,31147)(-23122,-12196)
\put(-1529,-5281){\makebox(0,0)[lb]{\smash{{\SetFigFont{20}{24.0}{\familydefault}{\mddefault}{\updefault}{\color[rgb]{0,0,0}$\tf{4}$}%
}}}}
\put(-8594,3404){\makebox(0,0)[lb]{\smash{{\SetFigFont{20}{24.0}{\familydefault}{\mddefault}{\updefault}{\color[rgb]{0,0,0}$\tc{\htwo}$}%
}}}}
\put(-8594,10649){\makebox(0,0)[lb]{\smash{{\SetFigFont{20}{24.0}{\familydefault}{\mddefault}{\updefault}{\color[rgb]{0,0,0}$\tc{\htwo}$}%
}}}}
\put(-18809,12404){\makebox(0,0)[lb]{\smash{{\SetFigFont{20}{24.0}{\familydefault}{\mddefault}{\updefault}{\color[rgb]{0,0,0}$\tc{\hone}$}%
}}}}
\put(-18809,1604){\makebox(0,0)[lb]{\smash{{\SetFigFont{20}{24.0}{\familydefault}{\mddefault}{\updefault}{\color[rgb]{0,0,0}$\tc{\hone}$}%
}}}}
\put(1576,1604){\makebox(0,0)[lb]{\smash{{\SetFigFont{20}{24.0}{\familydefault}{\mddefault}{\updefault}{\color[rgb]{0,0,0}$\tc{\hthree}$}%
}}}}
\put(1576,12404){\makebox(0,0)[lb]{\smash{{\SetFigFont{20}{24.0}{\familydefault}{\mddefault}{\updefault}{\color[rgb]{0,0,0}$\tc{\hthree}$}%
}}}}
\put(-7064,-11536){\makebox(0,0)[lb]{\smash{{\SetFigFont{20}{24.0}{\familydefault}{\mddefault}{\updefault}{\color[rgb]{0,0,0}$\tf{1}$}%
}}}}
\put(-14624,-5056){\makebox(0,0)[lb]{\smash{{\SetFigFont{20}{24.0}{\familydefault}{\mddefault}{\updefault}{\color[rgb]{0,0,0}$\tf{3}$}%
}}}}
\put(8686,12359){\makebox(0,0)[lb]{\smash{{\SetFigFont{20}{24.0}{\familydefault}{\mddefault}{\updefault}{\color[rgb]{0,0,0}$\tf{2}$}%
}}}}
\put(2161,8489){\makebox(0,0)[lb]{\smash{{\SetFigFont{20}{24.0}{\familydefault}{\mddefault}{\updefault}{\color[rgb]{0,0,0}$\tf{5}$}%
}}}}
\end{picture}%